\documentclass[11pt]{amsart}

\usepackage{geometry}
\usepackage{amsfonts, amssymb, amscd}
\numberwithin{equation}{section}

\usepackage[symbol]{footmisc}

\usepackage{bm}

\usepackage{verbatim}
\usepackage{amssymb}
\usepackage{mathrsfs}
\usepackage{graphicx}
\usepackage[nospace,noadjust]{cite}
\usepackage[all]{xy}
\usepackage{tikz-cd}
\usepackage{subcaption}
\usepackage{subfiles}
\usepackage[toc,page]{appendix}
\usepackage{mathtools}
\usepackage{comment}
\usepackage{enumerate}
\usepackage{enumitem}

\usepackage{graphicx}
\graphicspath{{images/}}

\usepackage{appendix}
\usepackage{hyperref}

\newcommand{\norm}[1]{\left\|{#1}\right\|}

\newcommand{\dashto}{\dashrightarrow}

\newcommand{\Spec}{\mathrm{Spec}}

\newcommand{\Oo}{\mathcal{O}}

\newcommand{\Pic}{\operatorname{Pic}}

\newcommand{\Cl}{\operatorname{Cl}}

\newcommand{\Supp}{\operatorname{Supp}}
\newcommand{\mult}{\operatorname{mult}}
\newcommand{\Center}{\operatorname{center}}
\newcommand{\Exc}{\operatorname{Exc}}
\newcommand{\Eff}{\operatorname{Eff}}

\newcommand{\Cc}{\mathbb{C}}

\newcommand{\Pp}{\mathbb{P}}
\newcommand{\Qq}{\mathbb{Q}}

\newcommand{\Rr}{\mathbb{R}}
\newcommand{\F}{\mathbb{F}}

\newcommand{\fR}{\mathfrak{R}}
\newcommand{\fD}{\mathfrak{D}}
\newcommand{\Zz}{\mathbb{Z}}

\newcommand{\cN}{\mathcal{N}}

\newcommand{\cD}{\mathcal{D}}
\newcommand{\cE}{\mathcal{E}}

\newcommand{\I}{\mathcal{I}}
\newcommand{\cZ}{\mathcal{Z}}

\newcommand{\bM}{\mathbf{M}}
\newcommand{\cX}{\mathcal{X}}
\newcommand{\cB}{\mathcal{B}}

\newcommand{\Rcomp}{\operatorname{Rct}}
\newcommand{\RCT}{\mathrm{\Rr CT}}

\newcommand{\vol}{\operatorname{vol}}

\newcommand{\WDiv}{\operatorname{WDiv}}

\newcommand{\lf}{\lfloor}
\newcommand{\rf}{\rfloor}

\newcommand{\Ii}{{\Gamma}}

 \usepackage{todonotes}

\newtheorem{thm}{Theorem}[section]

\newtheorem{lem}[thm]{Lemma}

\newtheorem{prop}[thm]{Proposition}

\newtheorem{claim}[thm]{Claim}
\theoremstyle{definition}
\newtheorem{defn}[thm]{Definition}
\newtheorem{rem}[thm]{Remark}

\newtheorem{ex}[thm]{Example}

\theoremstyle{definition}

\usepackage{todonotes}

\begin{document}

\title{Boundedness of complements for generalized pairs}

\author{Guodu Chen}
\address{School of Mathematical Sciences, Shanghai Jiao Tong University, 800 Dongchuan Road, Shanghai, 200240, China}
\email{chenguodu@sjtu.edu.cn}

\author{Jingjun Han}
\address{Shanghai Center for Mathematical Sciences, Fudan University, Shanghai, 200438, China}
\email{hanjingjun@fudan.edu.cn}

\author{Yang He}
\address{Beijing Institute of Mathematical Sciences and Applications, No. 544 Hefangkou Village Huaibei Town, Huairou District, Beijing, 101408, China}
\email{heyang@bimsa.cn}

\author{Lingyao Xie}
\address{Department of Mathematics, University of California, San Diego, 9500 Gilman Drive \# 0112, La Jolla, CA 92093-0112, USA}
\email{l6xie@ucsd.edu}

\begin{abstract}
We prove the boundedness of complements for Fano type generalized pairs (with the boundary coefficient set $[0,1]$) after Shokurov. 
\end{abstract}

\date{\today}

\maketitle
\pagestyle{myheadings}\markboth{\hfill G. Chen, J. Han, Y. He, and L. Xie \hfill}{\hfill Boundedness of complements for generalized pairs\hfill}

\tableofcontents
  
\section{Introduction}
We work over the field of complex numbers $\Cc.$ A \emph{variety} is an integral scheme that is separated and of finite type over $\Cc$.

Shokurov introduced the theory of complements (Definition \ref{defn:comp}) to investigate log flips for threefolds \cite{Sho92}. A key goal of this theory is to establish a certain boundedness property: roughly speaking, for any pair of a given dimension satisfying certain conditions, there exists an $n$-complement of the pair, where $n$ is a positive integer from a finite set that depends only on the dimension and is independent of the underlying variety \cite{Sho00, Sho04}. The theory of complements has had profound applications in the study of higher-dimensional varieties, including the Borisov-Alexeev-Borisov conjecture on the boundedness of Fano varieties \cite{Bir19, Bir21}, the M\textsuperscript{c}kernan-Shokurov conjecture on singularities on Fano type fibrations \cite{Bir23}, the Yau-Tian-Donaldson conjecture on the existence of Kähler-Einstein metrics on log Fano pairs \cite{LXZ21}, and Shokurov's ascending chain condition (ACC) conjecture for minimal log discrepancies \cite{HLL22, HLS19, HL22}. More recently, complement theory has been applied to the study of the boundedness and moduli of Calabi-Yau varieties \cite{ABB+23, Bir18}. For other applications in birational geometry, see, for example, \cite{BL23,CHL23,CHL24b,LS23}. It is widely believed that the boundedness of complements is deeply connected to other long-standing conjectures in birational geometry, particularly those related to the minimal model program, such as the termination of flips and Shokurov's ACC conjecture on $a$-lc thresholds and minimal log discrepancies. 

Much of the prior work has focused on a special case of boundedness for Fano-type varieties, specifically pairs where the boundary coefficients belong to a descending chain condition (DCC) set. After two decades of dedicated research, the boundedness of complements in this setting has been conclusively established \cite{Sho00, PS09, Bir19, HLS19}.

Shokurov's theory starts with $\Pp^1$, where he proved the boundedness of complements for arbitrary coefficients, meaning when the boundary coefficients belong to the interval $[0,1]$ \cite{Sho92}. The case of arbitrary coefficients is significantly more challenging than the DCC case. Notably, several important ACC-type results in birational geometry, such as the ACC for log canonical thresholds \cite{HMX14}, hold only for DCC coefficients and were pivotal in proving the boundedness of complements in this direction \cite{Bir21, HLS19}. Recently, Shokurov introduced new ideas and approaches, which ultimately led him to prove the boundedness of log canonical complements for Fano-type varieties with coefficients in $[0,1]$ \cite[Theorem 16]{Sho20}.

In this paper, following Shokurov's ideas, we show the boundedness of complements for generalized pairs with arbitrary boundary coefficients. In particular, when $\bM=0$, we provide a short and self-contained proof of Shokurov's result. 

\begin{thm}[=Theorem \ref{thm:klttype}]\label{thm:main}
Let $d$ and $p$ be two positive integers. Then there exists a finite set $\cN$ of positive integers depending only on $d$ and $p$ satisfying the following.

Assume that $(X/Z\ni z,B+\bM)$ is a g-pair of dimension $d$ such that $p\bM$ is b-Cartier, $X$ is of Fano type over $Z$, and $(X/Z\ni z,B+\bM)$ has an $\Rr$-complement which is klt over a neighborhood of $z$. Then $(X/Z\ni z,B+\bM)$ has an $n$-complement for some $n\in\cN$.
\end{thm}

We remark here that Theorem \ref{thm:main} fails if we remove the assumption of being \emph{klt over a neighborhood of $z$}. For a concrete example, see \cite[Example 11]{Sho20}. 
\medskip

Generalized pairs naturally arise in the study of birational geometry in higher dimensions. For instance, when applying induction on the dimension, generalized pairs appear in the canonical bundle formulas. They were first introduced in \cite{BZ16} by Birkar and Zhang in their work on the effective Iitaka fibration conjecture. Though seemingly technical, the theory of generalized pairs has proven to be a powerful tool in birational geometry. In particular, it has been a crucial component in the proof of the Borisov-Alexeev-Borisov conjecture \cite{Bir19, Bir21}. For recent progress, we refer readers to \cite{HL18, CHLX23, HL23}.

It is natural to consider the boundedness of complements for generalized pairs. In fact, if the boundary is required to belong to a hyperstandard set, the boundedness of complements for pairs and for projective generalized pairs are proved inductively \cite{Bir19}. Moreover, when the boundary coefficients of the generalized pairs belong to an arbitrary DCC set, the boundedness of complements is established in \cite{Chen23}, building on results from \cite{Bir19, HLS19}.

\smallskip

The following theorem plays a crucial role in proving Theorem \ref{thm:main}. 

\begin{thm}[cf.\ {\cite[Proposition 11]{Sho20}}]\label{thm:effisconst}
Let $f: X\to S$ be a projective morphism between normal quasi-projective varieties such that $X$ is of Fano type over $S$. Then possibly up to an \'etale base change, there is a canonical isomorphism $\Eff(X/S)\to\Eff(X_s)$ for any $s\in S$.
\end{thm}

We remark that if we replace the monoid $\Eff(X/S)$ with $\Eff(X/S)_\Qq$ or the cone $\Eff(X/S)_\Rr$, then the theorem typically follows from \cite[Lemma 2.7]{HX15}, which states that after an \'etale base change, 
$$
f_*\Oo_X\left(L^{\otimes m}\right)\to H^0\left(X_s,L^{\otimes m}|_{X_s}\right)
$$
is surjective for any line bundle $L$ and sufficiently divisible $m$ (where $m$ may depends on $L$). We actually prove something stronger: we show that after an \'etale base change,
$$
f_*\Oo_X(L)\to H^0\left(X_s,L|_{X_s}\right)
$$
is surjective for any Weil divisor $L$ and that $f_*\Oo_X(L)$ is locally free (note that $L|_{X_s}$ is well-defined when $X_s$ is normal). 

Our proof involves a subtle application of deformation invariance of log plurigenera \cite[Theorem 1.8]{HMX13}. One must be careful because they require log smoothness over the base, which usually leads to a shrinking of $S$ that depends on $L$, but we desire a common base for all $L$. Another technical issue arises since $L$ may not be Cartier, even if $L$ is $\Qq$-Cartier, it is not easy to construct a log smooth pair as in \cite[Theorem 1.8]{HMX13} to compute $H^0(X_s,L|_{X_s})$.

\medskip

\noindent\textbf{Sketch of the proof.}
We aim to prove Theorem \ref{thm:main} using induction on dimension, which involves several technical constructions, including the creation of log canonical centers, contractions to lower dimensional varieties, and lifting complements via ``effective'' adjunctions (Proposition \ref{prop:adjlift} and Proposition \ref{prop:cbflift1}). 
For convenience, we assume that $X$ is $\Qq$-factorial and $\dim Z=0$.

Firstly, we assume that $(X,B+G+\bM)$ has only klt singularities for any $0\le G\sim_{\Rr}-(K_X+B+\bM_X)$, i.e., $(X,B+\bM)$ is exceptional (see Definition \ref{defn:exc}). In this case, by \cite[Theorem 1.11]{Bir19}, we know that $X$ belongs to a bounded family. However, since the boundary $B$ may have arbitrary coefficients, we cannot even bound the number of components of $B$. This presents a significant challenge for us to show the boundedness of complements. Our strategy is to reduce our problem to cases where the number of components of $B$ is bounded. Specifically, we show that the monoid of effective Weil divisors modulo linear equivalence, $\Eff(X)$, has a finite number of generators, and more importantly, this number is uniformly bounded (see Theorem \ref{thm:effisconst} and Proposition \ref{prop:excgpairbdd}). As a result, we can assume that the boundary $B$ has a bounded number of components, for which the boundedness is already established. This leads to Theorem \ref{thm:exc}.

Secondly, we assume that $(X,B+G+\bM)$ has lc singularities for any $0\le G\sim_{\Rr}-(K_X+B+\bM_X)$, i.e., $(X,B+\bM)$ is semi-exceptional (see Definition \ref{defn:semiexc}). In this case, we aim to construct a contraction to a lower dimensional variety, and then apply the ``effective'' canonical bundle formula to lift bounded complements from the lower dimensional variety. For notation reason, we assume $\bM=0$ in this case. To apply traditional results (e.g. \cite{Bir19}), we first need to approximate $B$ by some hyperstandard set $\Gamma(\mathcal{N})$ (see Definition \ref{defn:nphi}), determined by a finite set $\mathcal{N}$ of positive integers. We use $B_{\mathcal{N}\_}$ to denote the best approximation of $B$ from below (see Definition \ref{defn:B_N}), where the coefficients of $B_{\mathcal{N}\_}$ belong to $\Gamma(\mathcal{N})$. The goal is to find a sufficiently large $\mathcal N$ such that $(X,B_{\mathcal{N}\_})$ is $\mathcal{N}$-complementary, i.e., there is an $n$-complement of $(Z,B_{\cN\_})$ for some $n\in \cN$. Then, by the delicate definition of $\Gamma(\mathcal N)$, $(X,B)$ would also be $\mathcal{N}$-complementary (cf. Lemma \ref{lem:N_Phicompl}). 

Let us start with the $\mathcal{N}^0$ given by Theorem \ref{thm:main} in lower dimensions (with $p=1$), then any canonical bundle formula for $(X,B_{\mathcal{N}^0\_})\to Z$ will give us a generalized pair with a DCC coefficient boundary $B^0_Z$ and a moduli part $\bM^0$ whose Cartier index is bounded by $p_0=p_0(d,\mathcal{N}^0)$. In fact, $p_0$ contains other information (see Proposition \ref{prop:cbfindex}) but we just ignore it for simplicity. Then if $p_0=1$, we can choose $\cN=\cN^0$ by lifting complements (Proposition \ref{prop:cbflift1}). Otherwise, we need a larger $\mathcal{N}^1=\mathcal{N}^1(p_0,\mathcal{N}^0)$ (given by induction hypothesis) such that $(Z,B^0_Z+\bM^0)$ is $\mathcal{N}^1$-complementary. Consequently, $(X,B_{\mathcal{N}^0\_})$ will be $\mathcal{N}^1$-complementary (but unfortunately, not $\mathcal{N}^0$-complementary) by lifting complements. We can then continue applying this argument and get a chain of finite sets 
$$\mathcal{N}^0 \subseteq\mathcal{N}^1 \subseteq \mathcal{N}^2 \subseteq \cdots \subseteq \mathcal{N}^i\subseteq\cdots.$$
The essential reason that $\cN^{i+1}$ could get larger each time is because $p_i> p_{i-1}$, but there are cases that the index could be the same, especially when we have a fibration $X\to Z$ which is both $K_X+B_{\cN^i\_}$-trivial and $K_X+B_{\cN^{i+1}\_}$-trivial. Therefore we are looking for such fibrations:

It is clear from the definition that 
$$\kappa\left(B-B_{\mathcal{N}^0\_}\right)\ge \kappa\left(B-B_{\mathcal{N}^{1}\_}\right)\ge\cdots\ge \kappa\left(B-B_{\mathcal{N}^{i}\_}\right),$$
and by Theorem \ref{thm:exc}, Lemma \ref{lem:semiexcnotbig} and some reductions we may assume that $\kappa\left(B-B_{\cN^0\_}\right)<d$. Thus, we can find
$$\kappa\left(B-B_{\mathcal{N}^i\_}\right)=\kappa\left(B-B_{\mathcal{N}^{i+1}\_}\right)=:k\le d-1$$
for some $i\le d$. 

After some extra effort and running some MMPs, we are able to obtain our desired fibration $\phi:X\to Z$, which is essential the same as the Iitaka fibration given by both $B-B_{\mathcal{N}^i\_}$ and $B-B_{\mathcal{N}^{i+1}\_}$. Comparing the canonical bundle formulas for $(X,B_{\cN^i\_})\to Z$ and $(X,B_{\cN^{i+1}\_})\to Z$, we see that the boundary parts differ by an effective $\Qq$-divisor from $Z$ (whose pullback is $B_{\mathcal{N}^{i+1}\_}-B_{\mathcal{N}^{i}\_}$), and the moduli parts are the same, denoted by $\bM_\phi$. In particular, the Cartier index of $\bM_\phi$ is bounded by $p_{i}$. By our construction of $\cN^{i+1}$, $(Z,B_{Z,i+1}+\bM_{\phi})$ is $\cN^{i+1}$-complementary, and again $(X,B_{\cN^{i+1}\_})$ would be $\cN^{i+1}$-complementary by lifting this complement to $X$. This leads to Theorem \ref{thm:semiexc}.


Finally, assume that $(X,B+G+\bM)$ is not lc for some $0\le G\sim_{\Rr}-(K_X+B+\bM_X)$. If $-(K_X+B+\bM_X)$ is big, then we can reduce to the case that $(X,B+\bM)$ is plt and $S:=\lf B\rf$ is a prime divisor on $X$ (see Lemma \ref{lem:pltcomp2}). By divisorial adjunction and induction, we can construct a bounded complement on $S$, and then lift it back to $X$ via effective divisorial adjunction (i.e., Proposition \ref{prop:adjlift}). This leads to Theorem \ref{thm:generictype}. In the general case, where $-(K_X+B+\bM_X)$ is no longer big, we apply the same arguments as in the proof of the semi-exceptional case. This leads to Theorem \ref{thm:klttype}.

To facilitate a clearer understanding of our proof logic, we outline here how the proofs of Theorem \ref{thm:main} (=Theorem \ref{thm:klttype}), Theorem \ref{thm:exc} (the exceptional case), Theorem \ref{thm:semiexc} (the semi-exceptional case), and Theorem \ref{thm:generictype} (the generic klt type) are conducted through induction:
\begin{itemize}
    \item Prove Theorem \ref{thm:exc}.
    \item Theorem \ref{thm:exc}$_{d}$ and Theorem \ref{thm:semiexc}$_{d-1}$ imply Theorem \ref{thm:semiexc}$_{d}$.
    \item Theorem \ref{thm:semiexc}$_{d}$ and Theorem \ref{thm:generictype}$_{d-1}$ imply Theorem \ref{thm:generictype}$_{d}$.
    \item Theorem \ref{thm:generictype}$_{d}$ and Theorem \ref{thm:klttype}$_{d-1}$ imply Theorem \ref{thm:klttype}$_{d}$.
\end{itemize}

\medskip

\noindent\textbf{Structure of the paper}.
In Section \ref{sec2}, we introduce some notation and tools that will be used in this paper. In Section \ref{sec3}, we recall the divisorial adjunction and the canonical bundle formulas for generalized pairs. In Section \ref{sec4}, we study the monoids of effective divisors and show Theorem \ref{thm:effisconst}. In Sections \ref{sec5}-\ref{sec7}, we prove Theorem \ref{thm:main} with additional conditions, i.e., Theorems \ref{thm:exc}, \ref{thm:semiexc} and \ref{thm:generictype}. In Section \ref{sec8}, we prove Theorem \ref{thm:main} completely. In Section \ref{sec9}, we provide a concise proof of \cite[Theorem 17]{Sho20} (Theorem \ref{thm:lctype}), which demonstrates that the klt assumption in Theorem \ref{thm:klttype} can be substituted with alternative arithmetic properties of the coefficients.

\bigskip

\noindent\textbf{Acknowledgement}.
The authors would like to thank Qianyu Chen, Christopher D. Hacon, Chen Jiang, Zhiyu Tian, and Qingyuan Xue for fruitful discussions. The first author was sponsored by the Natural Science Foundation of Shanghai (Grant No. 24ZR1430000) and the National Natural Science Foundation of China (Grant No. 12401055). The second author is affiliated with LMNS at Fudan University and has received support from the National Natural Science Foundation of China (Grant No. 12322102), and the National Key Research and Development Program of China (Grant No. 2023YFA1010600, No. 2020YFA0713200). The fourth author was partially supported by NSF research grants no. DMS-1801851, DMS-1952522, and DMS-58503413, as well as a grant from the Simons Foundation (Award Number: 256202). The fourth author is supported by another grant from the Simons Foundation.

\section{Preliminaries}\label{sec2}
For the definitions and the basic results on the singularities of pairs (resp. relative pairs) and the minimal model program, we refer the readers to \cite{KM98,BCHM10} (resp. \cite{CH21}). 

\subsection{Arithmetic of sets}\label{subsec:sets}

\begin{defn}[DCC and ACC sets]\label{defn:ACCDCCsets}
We say that a subset $\I\subseteq[0,1]$ satisfies the \emph{descending chain condition} (DCC) if any decreasing sequence $a_1\ge a_2\ge\dots$ in $\I$ stabilizes. We say that $\I$ satisfies the \emph{ascending chain condition} (ACC) if any increasing sequence $a_1\le a_2\le \dots$ in $\I$ stabilizes.

\smallskip

Suppose that $\fR\subseteq[0,1]\cap\Qq$ is a finite set, then we define
     $$\Phi(\fR):=\left\{1-\frac{r}{l}\le1\mid r\in\fR,l\in\Zz_{>0}\right\}$$
to be the set of \emph{hyperstandard multiplicities} associated to $\fR$ (cf. \cite[3.2]{PS09}). We may say that $\Phi(\fR)$ is the hyperstandard set associated to $\fR$. We usually assume that $0,1\in\fR$ without mention, and hence $\Phi(\{0,1\})\subseteq\Phi(\fR)$.
\end{defn}
The following notation is important throughout this paper.

\begin{defn}\label{defn:nphi}
Assume that $\cN\subseteq\Zz_{>0}$ is a finite set, and $\Phi:=\Phi(\fR)$ is the hyperstandard set associated to a finite set $\fR\subseteq[0,1]\cap\Qq$. We define
$$\Ii(\cN,\Phi):=\bigg\{1-\frac{r}{l}+\frac{1}{l}\sum_{n\in\cN}\frac{m_n}{n+1}\le1~\bigg|~ r\in\fR,l\in\Zz_{>0},m_n\in\Zz_{\ge0}\bigg\}.$$
For a positive integer $p$, by $p|\cN$, we mean that $p|n$ for any $n\in\cN$.
\end{defn}

By \cite[Remark 2.2]{CHX23}, we know that
\begin{itemize}
    \item $\Gamma(\cN,\Phi(\fR))$ is a hyperstandard set that is independent of the choice of $\fR$.
    \item If $\cN_1$ and $\cN_2$ are two finite sets of positive integers, then
$\Gamma(\cN_{1}\cup\cN_2,\Phi)=\Gamma(\cN_2,\Gamma(\cN_1,\Phi)).$
\end{itemize}
Thus for convenience, we write $\Gamma(\cN,\Phi)$ instead of $\Gamma(\cN,\Phi(\fR))$. For any $b\in[0,1]$, we define
$$b_{\cN\_\Phi}:=\max\left\{b'\mid b'\le b,\, b'\in\Ii(\cN,\Phi)\right\}.$$ 
If $\cN=\{n\}$ (respectively $\cN=\emptyset$), we may write $b_{n\_\Phi}$ (respectively $b_{\Phi}$) rather than $b_{\cN\_\Phi}$.

\begin{defn}\label{defn:B_N}
Assume that $\I\subseteq[0,1]$ is a subset and $B:=\sum b_iB_i$ is an $\Rr$-divisor on $X$, where $b_i\in\Rr$ and $B_i$ are prime divisors.
\begin{itemize}
    \item By $B\in\I$, we mean $b_i\in\I$ for any $i$.  If $B\in\Qq$ (resp. $B\in\Zz$), then we say $B$ is a $\Qq$-divisor (resp. divisor). If $B\ge0$, then we say $B$ is effective.
    \item We define 
$$\lf B\rf:=\sum\lf b_i\rf B_i,\  \{B\}:=\sum\{b_i\}B_i\text{, and }B^{>0}:=\sum_{b_i>0}b_iB_i.$$
  \item If $\cN\subseteq\Zz_{>0}$ is a finite set and $\Phi$ is a hyperstandard set, then we define
$$B_{\cN\_\Phi}:=\sum b_{\cN\_\Phi}B_i.$$
In particular, if $\cN=\{n\}$, we may write $B_{n\_\Phi}$ instead of $B_{\cN\_\Phi}$. If $\cN=\emptyset$, we may write $B_{\Phi}$ rather than $B_{\cN\_\Phi}$. 
\end{itemize}

\end{defn}

\subsection{Generalized pairs and singularities}
Here, we discuss very briefly the analogous concepts for generalized pairs, while for further details we refer the readers to \cite{BZ16,HL18}.

\begin{defn}
We say $\pi:X\to Z$ is a \emph{contraction} if $X$ and $Z$ are normal quasi-projective varieties, $\pi$ is a projective morphism, and $\pi_*\Oo_X = \Oo_Z$. 

A \emph{birational contraction} $\psi:X\dashto Y$ is a birational map such that $\psi^{-1}$ does not contract any divisor. Note that a birational contraction may not be a contraction.
\end{defn}
\begin{defn}
Let $\F$ be either the rational number field $\Qq$ or the real number field $\Rr$. Let $X$ be a normal variety and $\WDiv(X)$ the free abelian group of Weil divisors on $X$. Then an $\F$-divisor is defined to be an element of $\WDiv(X)_\F:=\WDiv(X)\otimes_\Zz\F$.

\smallskip

A \emph{b-divisor} on $X$ is an element of the projective limit
$$\textbf{WDiv}(X)=\lim_{Y\to X}\WDiv (Y),$$
where the limit is taken ober all the pushforward homomorphisms $f_*:\WDiv (Y)\to\WDiv(X)$ induced by proper birational morphisms $f:Y\to X$. In other words, a b-divisor $D$ on $X$ is a collection of Weil divisors $\textbf{D}_Y$ on higher models $Y$ of $X$ that are compatible under pushforward. The divisors $\textbf{D}_Y$ are called the \emph{traces} of $\textbf{D}$ on the birational models $Y.$ A b-$\F$-divisor is defined to be an element of $\textbf{WDiv}(X)\otimes\F$. 

\smallskip

The \emph{Cartier closure} of an $\F$-Cartier $\F$-divisor $D$ on $X$ is the b-$\F$-divisor $\overline{D}$ with trace $\overline{D}_Y=f^*D$ for any proper birational morphism $f:Y\to X$. A b-$\F$-divisor $\textbf{D}$ on $X$ is $\F$-Cartier if $\textbf{D}=\overline{D_{Y}}$ where $\overline{D_Y}$ is an $\F$-Cartier $\F$-divisor on a birational model over $X$, in this situation, we say $\textbf{D}$ \emph{descends} to $Y$. A b-$\F$-divisor if \emph{nef} if it descends to a nef $\F$-divisor on a birational model of $X$.
\end{defn}

\begin{defn}
A \emph{generalized sub-pair} (g-sub-pair for short) $(X/Z,B+\bM)$ consists of a projective morphism $X\to Z$ between normal quasi-projective varieties, an $\Rr$-divisor $B$ on $X$, and a nef over $Z$ b-$\Rr$-divisor $\bM$ on $X$, such that $K_X+B+\bM_X$ is $\Rr$-Cartier. If $B$ is effective, then we call $(X/Z,B+\bM)$ a \emph{generalized pair} (g-pair for short) and $B$ the \emph{boundary} of $(X/Z,B+\bM)$. If $Z$ is a point, then we may drop $Z$ from the notation and say that the generalized pair is \emph{projective}.

\smallskip

Let $ (X/Z,B+\bM) $ be a g-(sub-)pair and $f:W \to X$ is a log resolution of $(X,\Supp B)$ on which $\bM$ descends. We may write 
$$ K_{W} + B_W + \bM_W \sim_\Rr f^* ( K_X + B + \bM_X ) $$
for some $ \Rr $-divisor $ B_W$ on $ X' $, where we assume that $f_*K_W=K_X$. Let $ E $ be a prime divisor on $ W $. The \emph{log discrepancy} of $ E $ with respect to $ (X/Z,B+\bM) $ is defined as 
$$ a(E, X, B+\bM) :=1-\mult_E B_W . $$
We say $(X/Z,B+\bM)$ is \emph{$\epsilon$-lc} (resp. \emph{klt, lc}) for some non-negative real number $\epsilon$ if $a(E, X, B+\bM) \ge\epsilon$ (resp. $>0,\ \ge0$) for any prime divisor $E$ over $X$. If either $Z$ is clear from the context or $\dim Z=0$, then we may omit ``$Z$''.

For a prime divisor $E$ over $X$ with $a(E,X,B+\bM)\le0$, we call $E$ an \emph{lc place} and its image on $X$ an \emph{lc center}.

We say that $(Y/Z,B_Y+\bM)$ is a \emph{crepant model} of $(X/Z,B+\bM)$ if $(Y/Z,B_Y+\bM)$ and $(X/Z,B+\bM)$ are g-pairs, and there exists a birational morphism $f:Y\to X$ such that 
$$K_Y+B_Y+\bM_Y=f^*(K_X+B+\bM_X).$$
\end{defn}

\begin{defn}
A g-pair $ (X/Z,B+\bM) $ is called \emph{dlt} if it is lc, and there exists a closed subset $V \subseteq X$ such that the pair $(X \setminus V,B|_{X \setminus V})$ is log smooth, and if $ a(E,X,B+\bM) = 0$ for some prime $ E $ over $X$, then $\Center_X E \nsubseteq V $ and $ \Center_X E \setminus V $ is an lc center of $(X \setminus V,B|_{X \setminus V})$.
\end{defn}

By \cite[Proposition 3.10]{HL18}, any lc g-pair has a dlt modification.


\subsection{Fano type varieties}

\begin{defn}\label{defn:ftv}
Suppose that $X\to Z$ is a contraction. We say that $X$ is of \emph{Fano type} over $Z$ if $(X,B)$ is klt and $-(K_{X}+B)$ is big and nef over $Z$ for some boundary $B$.
\end{defn}
\begin{rem}
Assume that $X$ is of Fano type over $Z$. Then we can run a $D$-MMP over $Z$ for any $\Rr$-Cartier $\Rr$-divisor $D$ on $X$ and it terminates with either a good minimal model or a Mori fiber space (cf. \cite[Corollary 2.7]{PS09}). 
\end{rem}

\begin{lem}\label{lem:genericisFT}
Assume that $(X/Z,B+\bM)$ is a klt g-pair such that $K_X+B+\bM_X\sim_{\Rr,Z}0$ and $B+\bM_X$ is big over $Z$, then $X$ is of Fano type over $Z$. 
\end{lem}
\begin{proof}
We may write $B+\bM_X:=A+E$, for $\Rr$-divisors $E\ge0$ and $A$ which is ample over $Z$. Since $(X/Z,B+\bM)$ is klt, there exists a positive real number $\epsilon$, such that $(X/Z,(1-\epsilon)B+\epsilon E+(1-\epsilon)\bM)$ is klt. Note that 
$$-(K_X+(1-\epsilon)B+\epsilon E+(1-\epsilon)\bM_X)\sim_{\Rr,Z} \epsilon A.$$
By \cite[Lemma 3.5]{HL18}, there exists an $\Rr$-divisor $\Delta\sim_{\Rr,Z}(1-\epsilon)B+\epsilon E+(1-\epsilon)\bM_X+\frac{\epsilon}{2}A$ such that $(X,\Delta)$ is klt. Since $-(K_X+\Delta)\sim_{\Rr,Z}\frac{\epsilon}{2}A$ is ample over $Z$, we see that $X$ is of Fano type over $Z$.
\end{proof}

\begin{lem}\label{lem:S'isFT}
Let $(X/Z,B+\bM)$ be a $\Qq$-factorial plt g-pair such that $\lf B\rf=S$ is a prime divisor and $-(K_X+B+\bM_X)$ is ample over $Z$. Then $S\to\pi(S)$ is a contraction, where $\pi$ is the morphism $X\to Z$. In particular, $S$ is of Fano type over $\pi(S)$.
\end{lem}
\begin{proof}
From the following exact sequence
$$0\to\Oo_X(-S)\to\Oo_X\to\Oo_S\to0$$
we get the exact sequence 
$$\pi_*\Oo_X\to\pi_*\Oo_S\to{R}^1\pi_*\Oo_X(-S).$$
Since $-S=K_X+(B-S)+\bM_X-(K_X+B+\bM_X)$, $(X/Z,(B-S)+\bM)$ is klt and $-(K_X+B+\bM_X)$ is ample over $Z$, we have ${R}^1\pi_*\Oo_X(-S)=0$ by the relative Kawamata-Viehweg vanishing theorem for g-pairs \cite[Theorem 2.2.5]{CHLX23} (cf. \cite[Theorem 1.4]{CLX23}). It follows that $\pi_*\Oo_X\to\pi_*\Oo_S$ is surjective. Let $S\to V\to Z$ be the Stein factorization of $S\to Z$, then see that
$$\Oo_Z=\pi_*\Oo_X\to\pi_*\Oo_S=\mu_*\Oo_V,$$
is surjective, where $\mu$ is the morphism $V\to Z$. Since $\Oo_Z\to\mu_*\Oo_V$ factors as $\Oo_Z\to\Oo_{\pi(S)}\to\mu_*\Oo_V$, we know that $\Oo_{\pi(S)}\to\mu_*\Oo_V$ is a surjection. Hence $\Oo_{\pi(S)}\to\mu_*\Oo_V$ is an isomorphism as $\mu$ is finite. Therefore $V\to\pi(S)$ is an isomorphism and $S\to \pi(S)$ is a contraction.
\end{proof}

\subsection{Relative Iitaka dimensions of divisors}

\begin{defn}[Relative Iitaka dimensions, {\cite[Definition 2.1.1]{Cho08}}]
Let $X$ be a normal projective variety and $D$ an $\Rr$-Cartier $\Rr$-divisor on $X$. If $|\lf mD\rf|\neq0$ for some $m\in\Zz_{>0}$, then consider the dominant rational map
$$\phi_{|\lf mD\rf|}:X\dashto W_m,$$
with respect to the complete linear system of $\lf mD\rf$. We define the \emph{Iitaka dimension} of $D$ as
$$\kappa(D):=\max_{m>0}\{\dim W_m\},$$
if $H^0(X,\lf mD\rf)\neq0$ for some positive integer $m$, and $\kappa(D)=-\infty$ otherwise cf. \cite[II, 3.2, Definition]{Nak04}.

\smallskip

Let $f\colon  X\rightarrow Z$ be a projective morphism between varieties, $D$ an $\Rr$-divisor on $X$, and $F$ a very general fiber of the Stein factorialization of $f$. We define the \emph{relative Iitaka dimension} $\kappa(X/Z,D)$ of $D$ over $Z$ in the following way.
\begin{enumerate}
    \item $\kappa(D/Z):=\kappa(D|_F)$, if $|\lfloor mD\rfloor/Z|\not=\emptyset$ for some positive integer $m$ and $\dim F>0$.
    \item $\kappa(D/Z):=-\infty$, if $|\lfloor mD\rfloor/Z|=\emptyset$ for every positive integer $m$.
    \item $\kappa(D/Z):=0$, if $\dim F=0$.
\end{enumerate}
\end{defn}

\begin{lem}\label{lem:smiitdim}
Suppose that $X\to Z$ and $X'\to Z$ are two contractions. Assume that $X\dashto X'$ is a small birational contraction over $Z$, $D\ge0$ is a $\Qq$-Cartier $\Qq$-divisor on $X$ and $D'$ is the strict transform of $D$ on $X'$ which is also $\Qq$-Cartier. Then $\kappa(D/Z)=\kappa(D'/Z).$
\end{lem}
\begin{proof}
Let $f:W\to X$ and $f':W\to X'$ be a common resolution. We can write
$f^* D+E=f'^*D'+E'$ for some $f$-exceptional $\Qq$-divisor $E\ge0$ and $f'$-exceptional $\Qq$-divisor $E'$. Then 
\[\kappa(D/Z)=\kappa(f^* D+E/Z)=\kappa(f'^*D'+E'/Z)=\kappa(D'/Z).\qedhere\]
\end{proof}

\begin{lem}\label{lem:finitemmpiitdim}
Assume that $X\to Z$ is a contraction, $D_2\ge D_1$ are two $\Qq$-Cartier $\Qq$-divisors on $X$. Suppose that $X\dashto X'$ is a sequence of the $D_1$-MMP over $Z$, and $D_2'$ is the strict transform of $D_2$ on $X'$. Then $\kappa(D_2/Z)=\kappa(D_2'/Z)$.
\end{lem}
\begin{proof}
Pick a positive real number $\epsilon$ such that $X\dashto X'$ is also a sequence of the $(D_1+\epsilon D_2)$-MMP over $Z$. As $\Supp (D_1+\epsilon D_2)=\Supp(D_2)$, one can see that
$$\kappa(D_2'/Z)=\kappa(D_1'+\epsilon D_2'/Z)=\kappa(D_1+\epsilon D_2/Z)=\kappa(D_2/Z),$$
where $D_1'$ is the strict transform of $D_1$ on $X'$.
\end{proof}

\begin{lem}\label{lem:reltrivialmmp}
Assume that $X$ is of Fano type over $Z$. Assume that $\phi:X\to T$ is a contraction over $Z$, $H_T$ is an effective $\Qq$-divisor on $T$ that is ample over $Z$, and $D_2\ge D_1:=\phi^*H_T$ are two $\Qq$-Cartier $\Qq$-divisors such that $\kappa(D_2/Z)=\kappa(D_1/Z)$. Then we can run the $D_2$-MMP over $T$ which terminates with a model $X'$ such that $D_2'$ is semi-ample over $T$, where $D_2'$ is the strict transform of $D_2$ on $T'$. Moreover, if $\phi':X'\to T'$ is the induced morphism of $D_2'$ over $T$. Then $\tau:T'\to T$ is birational.
\end{lem}
\begin{proof}
Since $X$ is of Fano type over $Z$, the $D_2$-MMP over $T$ terminates. Moreover, by Lemma \ref{lem:finitemmpiitdim}, we have
$$\kappa(D'_2/Z)=\kappa(D_2/Z)=\kappa(D_1/Z)=\dim T-\dim Z.$$
By assumption, there exists a $\Qq$-divisor $H_{T'}$ on $T'$ which is ample over $T$ such that $D_2'=\phi'^*H_{T'}$ and $H_{T'}\ge\tau^*H_T$. Note that for any small enough positive real number $\epsilon$, $\tau^*H_T+\epsilon H_{T'}$ is ample over $Z$. Thus
\begin{align*}
\dim T'-\dim Z&=\kappa\left(\tau^*H_T+\epsilon H_{T'}/Z\right)=\kappa\left(H_{T'}/Z\right)=\kappa\left(D_2'/Z\right)=\dim T-\dim Z,
\end{align*}
that is, $\dim T'=\dim T$.
\end{proof}

\begin{lem}\label{lem:twommp}
Assume that $X$ is of Fano type over $Z$, and $D_2\ge D_1\ge0$ are two $\Qq$-Cartier $\Qq$-divisors on $X$ such that $\kappa(D_1/Z)=\kappa(D_2/Z)$. Then we have the following diagram 
 \begin{center}
	\begin{tikzcd}[column sep = 2em, row sep = 2em]
		X \arrow[d, "" swap]\arrow[rr, dashed]  && X' \arrow[d, "\pi'" swap] \arrow[rr, dashed] && X'' \arrow[d, "\pi''" swap] \\
		Z &&  \arrow[ll, ""]  Z'  && \arrow[ll, ""]Z''
	\end{tikzcd}
\end{center}
such that
\begin{enumerate}
  \item $X\dashto X'$ is an MMP over $Z$ on $D_1$, and $Z'$ is the canonical model over $Z$ of $D_1'$ ,
  \item $X'\dashto X''$ is an MMP over $Z'$ on $D_2'$, and $Z''$ is the canonical model over $Z'$ of $D_2''$ , and 
  \item $Z''\to Z'$ is a birational morphism,
\end{enumerate}
where $D_i'$ and $D_i''$ are the strict transforms of $D_i$ on $X'$ and $X''$ respectively for $i=1,2$. 
\end{lem}
\begin{proof}
Since $X$ is of Fano type over $Z$, the $D_1$-MMP over $Z$ and the $D_2'$-MMP over $Z'$ terminate respectively. By Lemma \ref{lem:finitemmpiitdim}, $\kappa(D_1'/Z)=\kappa(D_2'/Z)$. Since $Z'$ is the canonical model over $Z$ of $D_1'$, we may find an ample over $Z$ $\Qq$-divisor $H'_1$ on $Z'$ such that $D'_1=\pi'^*H'_1$. We conclude the statement by Lemma \ref{lem:reltrivialmmp}.
\end{proof}

\subsection{Complements}

\begin{defn}[Complements]\label{defn:comp}
Let $(X/Z\ni z, B+\bM)$ be a g-pair. We say that a g-pair $(X/Z\ni z,B^++\bM)$ is an \emph{$\Rr$-complement} of $(X/Z\ni z,B+\bM)$ if over a neighborhood of $z$, $(X/Z,B^++\bM)$ is lc, $B^+\ge B$, and $K_X+B^++\bM_X\sim_\Rr0$. 

Let $n$ be a positive integer. We say that a g-pair $(X/Z\ni z,B^++\bM)$ is an \emph{$n$-complement} of $(X/Z\ni z,B+\bM)$, if over a neighborhood of $z$, we have
\begin{enumerate}
  \item $(X/Z,B^++\bM)$ is lc, 
  \item $nB^+\ge n\lf B\rf+\lf(n+1)\{B\}\rf$, and
  \item $n(K_X+B^++\bM_X)\sim 0$ and $n\bM$ is b-Cartier.
\end{enumerate}
Moreover, $(X/Z\ni z,B^++\bM)$ is \emph{monotonic} if we additionally have $B^+\ge B$.
\end{defn}
We say that $(X/Z\ni z,B+\bM)$ is \emph{$\Rr$-complementary} (resp. \emph{$n$-complementary}) if it has an $\Rr$-complement (resp. $n$-complement). Assume that $\cN\subseteq\Zz_{>0}$ is a finite set, we say $(X/Z\ni z,B+\bM)$ is \emph{$\cN$-complementary} if it is $n$-complementary for some $n\in\cN$. If $\dim Z=\dim z=0$, we will omit $Z$ and $z$.


\begin{lem}\label{lem:Qcompl}
Suppose that $(X/Z\ni z,B+\bM)$ is a (klt) $\Rr$-complementary g-pair. If $B\in\Qq$ and $\bM$ is $\Qq$-Cartier. Then $(X/Z\ni z,B+\bM)$ has a (klt) $\Rr$-complement $(X/Z\ni z,B^++\bM)$ such that $B^+\in\Qq$.
\end{lem}
\begin{proof}
Let $(X/Z\ni z,\tilde{B}+\bM)$ be an $\Rr$-complement of $(X/Z\ni z,B+\bM)$. By \cite[Lemma 4.1]{HanLiu20}, there is a $\Qq$-divisor $B^+$ such that $B^+\ge B$ such that $(X/Z\ni z,B^++\bM)$ is an $\Rr$-complement of $(X/Z\ni z,B+\bM)$. Moreover, if $(X/Z\ni z,\tilde{B}+\bM)$ is klt, then we may pick $(X/Z\ni z,B^++\bM)$ to be klt.
\end{proof}

The following lemma is well-known to experts (cf. \cite[6.1]{Bir19}). We will use the lemma frequently without citing it in this paper. 
\begin{lem}
Let $(X/Z\ni z,B+\bM)$ be a g-pair. Assume that $g:X\dashto X'$ is a birational contraction over $Z$ and $B'$ is the strict transform of $B$ on $X'$.
\begin{enumerate}
 \item If $(X/Z\ni z,B+\bM)$ is $\Rr$-complementary, then $(X'/Z\ni z,B'+\bM)$ is $\Rr$-complementary.
 \item Let $n$ be a positive integer. If $g$ is $-(K_X+B+\bM_X)$-non-positive and $(X'/Z\ni z,B'+\bM)$ is $\Rr$-complementary (resp. monotonic $n$-complementary), then $(X/Z\ni z,B+\bM)$ is $\Rr$-complementary (resp. monotonic $n$-complementary).
\end{enumerate}
\end{lem}

\begin{lem}\label{lem:N_Phicompl}
Assume that $\cN\subseteq\Zz_{>0}$ is a finite set, $\Phi$ is a hyperstandard set associated to a finite set $\fR\subseteq\Qq\cap[0,1]$, and $n$ is a positive integer such that $n\fR\subseteq\Zz_{\ge0}$. Assume that $(X/Z\ni z,B+\bM)$ is a g-pair.
\begin{enumerate}
    \item If $(X/Z\ni z,B_{\cN\_\Phi}+\bM)$ is $\cN$-complementary, then $(X/Z\ni z,B+\bM)$ is $\cN$-complementary.
    \item Any $n$-complement of $(X/Z\ni z,B+\bM)$ is a monotonic $n$-complement of $(X/Z\ni z,B_{n\_\Phi}+\bM)$.
\end{enumerate}
\end{lem}
\begin{proof}
The lemma follows from \cite[Lemmas 2.3 and 2.4]{CHX23}.
\end{proof}

\begin{lem}\label{lem:Rcompgpairimplypair}
Let $\epsilon<1$ be a non-negative real number. Let $(X,B+\bM)$ be a g-pair which has an $\epsilon$-lc $\Rr$-complement. Suppose that $X$ is of Fano type and $K_X+B$ is $\Rr$-Cartier. Then $(X,B)$ has an $\Rr$-complement which is $\epsilon$-lc. 
\end{lem}
\begin{proof}
By assumption, we may run a $-(K_X+B)$-MMP which terminates with a model $X'$ such that $-(K_{X'}+B')$ is nef and $(X',B')$ is $\epsilon$-lc, where $B'$ is the strict transform of $B$ on $X'$. In particular, $(X',B')$ has an $\Rr$-complement which is $\epsilon$-lc and the lemma follows.
\end{proof}

\begin{defn}[$\Rr$-complementary thresholds]\label{defn:rct} 
Suppose that $(X/Z\ni z, B+\bM)$ is an $\Rr$-complementary g-pair and $D$ is an $\Rr$-Cartier $\Rr$-divisor. The \emph{$\Rr$-complementary threshold} of $D$ with respect to $(X/Z\ni z,B+\bM)$ is defined as
     \begin{align*}
     \Rcomp(X/Z\ni z,B+\bM;D):=\sup\{t\ge0\mid(X/Z\ni z,B+tD+\bM)\text{ is $\Rr$-complementary}\}.
     \end{align*}
     In particular, if $X=Z$, $\pi:X\to Z$ is the identity map, then we obtain the lc threshold.
\end{defn}

We have the following ACC property of $\Rr$-complementary thresholds. 

\begin{thm}\label{thm:rctacc}
Let $d,p$ be two positive integers, and $\Ii$ a DCC set. Then the set 
\[
\RCT(d,p,\Ii):=\left\{\Rcomp(X/Z\ni z,B+\bM;D) ~\Bigg|~ 
\begin{aligned}
   & \dim X=d,B,D\in\Ii,p\bM\text{ is b-Cartier},\\
  &X\text{ is of Fano type over $Z$}, \text{ and}\\&
   (X/Z\ni z,B+\bM)\text{ is $\Rr$-complementary}
\end{aligned}
\right\}
\]
satisfies the ACC.     
\end{thm}
\begin{proof}
Suppose on the contrary, there exists a sequence g-pairs $(X_i/Z_i\ni z_i,B_i+\bM_i)$ of dimension $d$ and $\Rr$-Cartier divisors $D_i$, such that $(X_i/Z_i\ni z_i,B_i+\bM_i)$ is $\Rr$-complementary, $B_i,D_i\in \Ii$, $p\bM_i$ are b-Cartier, and 
$$t_i:=\Rcomp(X_i/Z_i\ni z_i,B_i+\bM_i;D_i)$$ 
is strictly increasing. In particular, the coefficients of $B_i+t_iD_i$ belong to a DCC set $\Ii_1$. By \cite[Theorem 5.2]{Chen23}, there exists an $\Rr$-divisor $B'_i\ge B_i+t_iM_i$, such that the coefficients of $B_i'$ belong to a finite set $\Ii'$, and $(X_i/Z_i\ni z_i,B_i'+\bM_i)$ is $\Rr$-complementary. By the construction of $t_i$, there exists a prime divisor $S_i$, such that $\Supp S_i\subseteq\Supp D_i$ and 
$$\mult_{S_i} (B_i+t_iD_i)=\mult_{S_i}B_i'\in \Ii'$$
 for any $i$. Thus $t_i$ belongs to an ACC set, a contradiction.
\end{proof}

\section{Adjunction}\label{sec3}
In this section, we study the divisorial adjunction and canonical bundle formula for g-pairs cf. \cite{Bir19}. Moreover, we will consider lifting complements from lower dimensional varieties via divisorial adjunction and canonical bundle formula respectively.

\subsection{Divisorial adjunction for generalized pairs}
\begin{defn}[Divisorial adjunction]
Let $(X/Z,B+\bM)$ be a g-pair and $f:W\to X$ a log resolution of $(X/Z,B)$ such that $\bM=\overline{\bM_{W}}$. Assume that $S$ is the normalization of a component of $\lf B\rf$, and that $S_W$ is its strict transform on $W$. We may write
$$ K_{W}+B_W+\bM_{W}:=f^*(K_{X}+B+\bM_X),$$
for some $\Rr$-divisor $B_W.$ Let $B_{S_W}:=(B_W-S_W)|_{S_W}$, $\bM^\iota_{S_W}:=\bM_{W}|_{S_W}$, and $\bM^\iota:=\overline{\bM^\iota_{S_W}}$, then we get
$$K_{S_W}+B_{S_W}+\bM^\iota_{S_W}=f^*(K_{X}+B+\bM_X)|_S.$$
Let $f_{S_W}:=f|_{S_W}$ be the induced morphism and $B_{S}:=(f_{S_W})_*B_{S_W}$.
Then we have
$$K_{S}+B_{S}+\bM^\iota_{S}=(K_X+B+\bM_X)|_S,$$
and $(S/Z,B_S+\bM^\iota)$ is a g-pair which we refer to as \emph{divisorial adjunction}.
\end{defn}

Note that if $(X/Z,B+\bM)$ is lc and $\bM$ is $\Qq$-Cartier, then $(S/Z,B_S+\bM^\iota)$ is also lc. Conversely, if $(X/Z,S+\bM)$ is plt, $(S/Z,B_S+\bM^\iota)$ is lc, and $\bM$ is $\Qq$-Cartier, then $(X/Z,B+\bM)$ is lc near $S$ (\cite[Theorem 4.4.1]{CHLX23}).

\smallskip

The following result is a relative version of \cite[Lemma 3.3]{Bir19}. 
\begin{prop}\label{prop:adjcoeff}
Let $p$ be a positive integer and $\Phi$ the hyperstandard set associated to a finite set $\fR\subseteq[0,1]\cap\Qq$. Then there exists a hyperstandard set $\Phi'$ depending only on $p$ and $\Phi$ satisfying the following.

Assume that $(X/Z,B+\bM)$ is an lc g-pair of dimension $d$ such that
\begin{enumerate}
   \item $B\in\Phi$ and $p\bM$ is b-Cartier,
   \item $S$ is the normalization of a component of $\lf B\rf$, and 
   \item $(S/Z,B_{S}+\bM^\iota)$ is the g-pair induced by divisorial adjunction.
\end{enumerate} 
Then $B_S\in\Phi'$. Moreover, if we replace $\Phi$ by $\Ii(n,\Phi)$ for some positive integer $n$, then $B_S\in\Ii(n,\Phi')$.
\end{prop}
\begin{proof}
We show that the hyperstandard set $\Phi'$ associated to the following finite set
$$\fR':=\bigg\{1-\sum(1-r_i)\ge0~\bigg|~ r_i\in\fR\cup\left\{\frac{p-1}{p}\right\}\bigg\}$$ 
has the required properties. Indeed, by the same arguments as in the proof of \cite[Lemma 3.3]{Bir19}, we have $B_{S}\in\Phi'$. Hence we only need to prove the ``Moreover'' part. In fact, let
$$\fR_1:=\left\{r-\frac{m}{n+1}\ge0~\bigg|~ r\in\fR,m\in\Zz_{\ge0}\right\},$$
i.e., $\Phi(\fR_1)=\Ii(n,\Phi)$. By \cite[Lemma 3.3]{Bir19} again, the coefficients of $B_S$ belong to the hyperstandard set $\Phi_1'$ associated to the following finite set of rational numbers
\begin{align*}
\fR_1':&=\bigg\{1-\sum(1-r_i')\ge0~\bigg|~ r_i'\in\fR_1\cup\left\{\frac{p-1}{p}\right\}\bigg\}\\
&=\bigg\{1-\sum(1-r_i)-\frac{m}{n+1}~\bigg|~ r_i\in\fR\cup\left\{\frac{p-1}{p}\right\},m\in\Zz_{\ge0}\bigg\}.
\end{align*}
We conclude the statement as $\Phi_1'=\Ii(n,\Phi')$.
\end{proof}

\subsection{Canonical bundle formula for generalized pairs}
By the work of Kawamata \cite{Kaw97,Kaw98} and Ambro \cite{Amb05}, we have the so-called \emph{canonical bundle formula} for usual pairs. In the following, we recall the canonical bundle formula for g-pairs, cf. \cite{Bir19,Fil20,CHLX23}.

Suppose that $(X/Z,B+\bM)$ is an lc g-pair and $\phi:X\to T$ is a contraction over $Z$ such that $\dim T>0$ and $K_X+B+\bM_X\sim_{\Rr,T}0$. 
Then we can find a uniquely determined $\Rr$-divisor $B_T$ and a nef$/Z$ b-$\Rr$-divisor $\bM_{\phi}$ which is determined only up to $\Rr$-linear equivalence such that $(T/Z,B_T+\bM_{\phi})$ is lc and
$$K_X+B+\bM_X\sim_{\Rr}\phi^*(K_T+B_T+\bM_{\phi,T}).$$
Here $B_{T}$ (resp. $\bM_{\phi}$) is called the \emph{discriminant part} (resp. a \emph{moduli part}) of the canonical bundle formula for $(X/Z,B+\bM)$ over $T$. Moreover, if $(X/Z,B+\bM)$ is klt, then so is $(T/Z,B_T+\bM_{\phi})$.


\begin{lem}\label{lem:cbfcptoverbase}
Notation as above. Assume that $\bM$ is $\Qq$-Cartier. Then there exists a crepant model $(\tilde{T},B_{\tilde{T}}+\bM_{\phi}) \to (T,B_T+\bM_\phi)$ such that for any prime divisor $P\subseteq\Supp B$ which is vertical over $T$, the image of $P$ on $\tilde{T}$ is a prime divisor.
\end{lem}
\begin{proof}
According to \cite[Theorem 0.3]{AK00}, there exist birational morphisms $X'\to X$ and $T'\to T$ such that $X'\to T'$ is an equidimensional contraction. In particular, for any prime divisor $P\subseteq\Supp B$ which is vertical over $T$, the image $Q$ of $P$ on $T'$ is a prime divisor. Moreover, by the canonical bundle formula, $a(Q,T,B_T+\bM_\phi)<1$ as $a(P,X,B)<1$. Then the lemma holds by \cite[Theorem 1.7]{LX23}.
\end{proof}

The following proposition is the relative version of \cite[Proposition 6.3]{Bir19} in the generalized pair setting. 

\begin{prop}[cf. {\cite[Proposition 3.3]{HJ24}}]\label{prop:cbfindex}
Let $d,p$ be two positive integers and $\Phi$ the hyperstandard set associated to a finite set $\fR\subseteq[0,1]\cap\Qq$. Then there exist a positive integer $p'$ and a hyperstandard set $\Phi'\subseteq[0,1]$ depending only on $d,p,$ and $\Phi$ satisfying the following.

Assume that $(X/Z,B+\bM)$ is a g-pair and $\phi:X\to T$ is a contraction over $Z$ such that
\begin{enumerate}
  \item $(X/Z,B+\bM)$ is lc of dimension $d$, and $\dim T>0$,
  \item $X$ is of Fano type over $T$, 
  \item $B\in\Phi$ and $p\bM$ is b-Cartier, and
  \item $K_X+B+\bM_X\sim_{\Qq,T}0$.
\end{enumerate}
Then we can choose a moduli part $\bM_\phi$ of the canonical bundle formula for $(X/Z,B+\bM)$ over $T$ such that $B_T\in\Phi'$, $p'\bM_\phi$ is b-Cartier, and
$$p'(K_X+B+\bM_X)\sim p'\phi^*(K_T+B_T+\bM_{\phi,T}),$$
where $B_T$ is the discriminant of the canonical bundle formula for $(X/Z,B+\bM)$ over $T$.
\end{prop}
\begin{proof}
The proof follows from the same arguments as in \cite[Proposition 6.3]{Bir19} (see also \cite[Proposition 3.1]{CHL23}) but utilizes the boundedness of relative complements for g-pairs \cite[Theorem 1.1]{Chen23}.
\end{proof}

\subsection{Lifting complements}
In this subsection, we lift complements from lower dimensional varieties via divisorial adjunction and canonical bundle formula respectively.

\begin{prop}\label{prop:adjlift}
Let $n$ be a positive integer and $(X/Z\ni z,B+\bM)$ a plt g-pair such that 
\begin{enumerate}
  \item $n\bM$ is b-Cartier,
  \item $-(K_X+B+\bM_X)$ is big and nef over $Z$,
  \item $\lf B\rf=S$ is a prime divisor such that $\Center_ZS\ni z$, and
  \item $(S,B_S+\bM^\iota)$ has a monotonic $n$-complement $(S,B_S^++\bM^\iota)$ over a neighborhood of $z$, where $K_S+B_S+\bM^\iota_S:=(K_X+B+\bM_X)|_S$.
\end{enumerate}
Then $(X/Z\ni z,B+\bM)$ has an $n$-complement $(X/Z\ni z,B^++\bM)$ such that $(K_X+B^++\bM_X)|_S=K_S+B_S^++\bM^\iota_S.$
\end{prop}
\begin{proof}
Possibly after shrinking $Z$ near $z$, we may assume that $(S,B_S^++\bM^\iota)$ is a monotonic $n$-complement of $(S,B_S+\bM^\iota)$. Let $f:W\to X$ be a log resolution of $(X,B)$, such that $\bM=\overline{\bM_W}$, and $S_W$ the strict transform of $S$ on $W$. Set $f_{S_W}:=f|_{S_W}$.
We may write
$$K_W+B_W+\bM_W:=f^{*}(K_X+B+\bM_X),$$
and
$$n(K_{S_W}+B_{S_W}^{+}+\bM^\iota_{S_W}):=nf_{S_W}^{*}(K_S+B_S^{+}+\bM^\iota_S)\sim 0,$$
for some $\Rr$-divisors $B_W$ and $B_{S_W}^+$. Let
$$L_W:=\lceil -(n+1)(K_W+B_W)\rceil.$$
Note that $\lf B_W^{\ge0}\rf=S_W$ as $(X/Z,B+\bM)$ is plt, and we have
$$K_{S_W}+B_{S_W}+\bM^\iota_{S_W}:=(K_W+B_W+\bM_W)|_{S_W}=f_{S_W}^{*}(K_S+B_S+\bM^\iota_S).$$
Since $-(n+1)(K_W+B_W+\bM_W)+\bM_W$ is big and nef over $Z,$ $\{(n+1)(K_W+B_W)\}$ is snc, and it follows that
\begin{align*}
L_W-n\bM_W&=\lceil-(n+1)(K_W+B_W)\rceil-n\bM_W\\&=-(n+1)(K_W+B_W+\bM_W)+\bM_W
+\{(n+1)(K_W+B_W)\},
\end{align*}
by the relative Kawamata-Viehweg vanishing theorem, we have
$$R^1h_{*}(W,\mathcal{O}_{W}(K_W+L_W-n\bM_W))=0,$$ 
where $h$ is the induced morphism $W\to Z$. From the exact sequence
\begin{align*}
0\to\mathcal{O}_W(K_W+L_W-n\bM_W)&\to \mathcal{O}_W(K_W+S_W+L_W-n\bM_W)\\&\to \mathcal{O}_{S_W}(K_{W}+S_W+L_W-n\bM_W)\to 0,
\end{align*}
we deduce that the induced morphism
\begin{align}\label{surj}
H^0(W,K_W+S_W+L_W-n\bM_W)\to H^0(S_W,(K_{W}+S_W+L_W-n\bM_W)|_{S_W})
\end{align}
is surjective. 

\smallskip

Note that
\begin{align}\label{onW}
K_W+S_W+L_W-n\bM_W=&-nK_W-nS_W-\lf(n+1)(B_W-S_W)\rf-n\bM_W.
\end{align}
Since $f$ is a log resolution of $(X,B)$, and $nB_{S_W}^{+}\ge\lf (n+1)B_{S_W}\rf$, one has
\begin{align*}
&\ \ \ \ (K_{W}+S_W+L_W-n\bM_W)|_{S_W}\\
&=(-nK_W-nS_W-\lf(n+1)(B_W-S_W)\rf-n\bM_W)|_{S_W}\\
&=-nK_{S_W}-\lf(n+1)B_{S_W}\rf-n\bM_{S_W}^\iota\\
&=-n(K_{S_W}+B_{S_W}^{+}+\bM^\iota_{S_W})+nB_{S_W}^{+}-\lf(n+1)B_{S_W}\rf
\\
&\sim D_{S_W}:=nB_{S_W}^{+}-\lf(n+1)B_{S_W}\rf\ge0.
\end{align*}
It follows from \eqref{surj} that we may find an effective divisor $D_W$ on $W$ such that 
$$D_W\sim K_W+S_W+L_W-n\bM_W,$$ 
and $D_W|_{S_W}=D_{S_W}$. Let $D:=f_*D_W$,
$$B_W^+:=S_W+\frac{\lf(n+1)(B_W-S_W)\rf+D_W}{n}$$
and 
$$B^+:=f_*B_W^+=S+\frac{\lf (n+1)\{B\}\rf+D}{n}.$$
By \eqref{onW} and the negativity lemma, we infer that
$$n(K_W+B_W^++\bM_W)=nf^*(K_X+B^++\bM_X)\sim0.$$
Moreover, it is easy to see that
\begin{align*}
&\ \ \ \ (K_W+B_W^++\bM_W)|_{S_W}\\
&=(K_W+S_W+\frac{\lf(n+1)(B_W-S_W)\rf+D_W}{n}+\bM_W)|_{S_W}\\
&=K_{S_W}+\frac{\lf(n+1)B_{S_W}\rf+D_{S_W}}{n}+\bM_{S_W}^\iota\\
&=K_{S_W}+B_{S_W}^++\bM_{S_W}^\iota.
\end{align*}
Therefore we conclude that 
$$(K_X+B^++\bM_X)|_{S}=K_S+B_S^++\bM_S^\iota.$$

It suffices to show that $(X/Z,B^++\bM)$ is lc over a neighborhood of $z$. Suppose on the contrary that $(X/Z,B^{+}+\bM)$ is not lc over a neighborhood of $z$, then there exists a real number $a\in(0,1)$, such that $(X/Z,aB^{+}+(1-a)B+\bM)$ is not lc near the fiber over $z$. On the other hand, by the inversion of adjunction, $(X/Z,aB^{+}+(1-a)B+\bM)$ is lc near $S$. Thus the non-klt locus of $(X/Z,aB^{+}+(1-a)B+\bM)$ near the fiber over $z$ has at least two disjoint components one of which is $S$. This contradicts Shokurov-Koll\'ar's connectedness principle for generalized pairs \cite[Lemma 2.14]{Bir19} as
$$-(K_X+aB^{+}+(1-a)B+\bM_X)=-a(K_X+B^{+}+\bM_X)-(1-a)(K_X+B+\bM_X)$$
is big and nef over $Z$. Therefore $(X/Z,B^{+}+\bM)$ is lc over a neighborhood of $z$.  
\end{proof}
\begin{rem}
In Proposition \ref{prop:adjlift}, $(X/Z\ni z,B^++\bM)$ may not be a monotonic $n$-complement of $(X/Z\ni z,B+\bM)$. If $B\in\Ii(n,\Phi)$ for some hyperstandard set $\Phi$ associated to a finite set $\fR\subseteq[0,1]\cap\frac{1}{n}\Zz$, then $B^+\ge B$, see Lemma \ref{lem:N_Phicompl}.
\end{rem}

Now we turn to the following technical statement on lifting complements by canonical bundle formula. It plays an important role in the rest of the paper.

\begin{prop}\label{prop:cbflift1}
Let $p$ and $n$ be two positive integers such that $p|n$. Assume that $(X/Z,B+\bM)$ is a g-pair and $\phi:X\to T$ is a contraction over $Z$ such that
\begin{enumerate}
  \item $\dim T>0$,
  \item $B$ is a $\Qq$-divisor,
  \item $p\bM$ and $p\bM_{\phi}$ are b-Cartier,
  \item $p(K_{X}+B+\bM_{X})\sim p\phi^*(K_{T}+B_{T}+\bM_{\phi,T})$, and
  \item $(T/Z,B_T+\bM_{\phi})$ has a crepant model $(T'/Z,B_{T'}+\bM_{\phi})$ such that $\Center_{T'}P$ is a prime divisor for any prime divisor $P\subseteq\Supp B$ that is vertical over $T$,
\end{enumerate}
where $B_{T}$ and $\bM_{\phi}$ are the discriminant and moduli parts of canonical bundle formula for $(X/Z,B+\bM)$ over $T$. If $(T'/Z\ni z,B_{T'}+\bM_{\phi})$ is $n$-complementary, then $(X/Z\ni z,B+\bM)$ is $n$-complementary.
\end{prop}
\begin{proof}
Let $X'$ be the normalization of the main component of $X\times_{T}T'$, then we have the following commutative diagram
\begin{center}
	\begin{tikzcd}[column sep = 2em, row sep = 2em]
		X \arrow[d, "\phi" swap]  && \arrow[ll,"f" swap] X' \arrow[d, "\phi'" swap] \\
		T &&  \arrow[ll, ""]  T'  .
	\end{tikzcd}
\end{center}
We may write 
$$K_{X'}+B'+\bM_{X'}=f^*(K_X+B+\bM_X)$$
for some $\Rr$-divisor $B'$. Note that by our assumption, we have
$$p(K_{X'}+B'+\bM_{X'})\sim p\phi'^*(K_{T'}+B_{T'}+\bM_{\phi,T'}).$$

\smallskip

By assumption, $(T'/Z\ni z,B_{T'}+\bM_{\phi})$ has an $n$-complement $(T'/Z\ni z,B_{T'}^++\bM_{\phi})$. Possibly after shrinking $Z$ near $z$, we may assume that
$$n(K_{T'}+B_{T'}^++\bM_{\phi,T'})\sim_Z0.$$ 
Let $B'^+:=B'+\phi'^*(B_{T'}^+-B_{T'})$ and $B^+:=f_*B'^+$. We claim that $(X/Z\ni z,B^++\bM)$ is an $n$-complement of $(X/Z\ni z,B+\bM)$. Indeed, we have
\begin{align*}
&\ \ \ \ n(K_{X'}+B'^++\bM_{X'})=n(K_{X'}+B'+\bM_{X'})+n(B'^+-B')\\&
\sim n\phi'^*(K_{T'}+B_{T'}+\bM_{\phi,T'})+n\phi'^*(B_{T'}^+-B_{T'})\\& = n\phi'^*(K_{T'}+B^+_{T'}+\bM_{\phi,T'})\sim_Z0.
\end{align*}
Hence $n(K_X+B^++\bM_X)\sim_Z0$. By \cite[Lemma 11.4.3]{CHLX23}, $(X',B'^{+}+\bM)$ is sub-lc and thus so is $(X,B^++\bM_X)$.
It suffices to prove that 
$$nB^+\ge \lf(n+1)\{B\}\rf+n\lf B\rf.$$

\smallskip

For any prime divisor $P\subseteq\Supp B$, let $Q$ be the image of $P$ on $T'$ which is a prime divisor. Without loss of generality, we may assume that $Q$ is a Cartier divisor. Then let $b_P:=\mult_PB$, $b_P^+:=\mult_PB^+,b_Q:=\mult_QB_{T'}$, $b_Q^+:=\mult_QB_{T'}^+$ and $m_Q:=\mult_P\phi'^*Q$. By construction, 
$$b_P^+=b_P+(b_Q^+-b_Q)m_Q.$$ 
Hence
$$r_{PQ}:=b_P+(1-b_Q)m_Q=b_P^++(1-b_Q^+)m_Q\in\frac{1}{n}\Zz_{>0}.$$
Moreover, as $1-b_Q$ is the lc threshold of $Q$ with respect to $(X/Z,B+\bM)$ over the generic point of $Q$, we know $r_{PQ}\le1$. If $b_Q=1$, then $b_Q^+=1$ which follows that $b_P=b_P^+.$ If $b_P=1$, then $r_{PQ}=b_Q=1$ and thus $b_P^+=1$. Hence we may assume that $b_Q<1$ and $b_P<1$. 
Since 
$$b_P=b_P^+- (b_Q^+-b_Q)m_Q\text{ and }nb_Q^+\ge\lf(n+1)b_Q\rf,$$ 
we can see that
\begin{align*}
\lf(n+1)b_P\rf&=\lf(n+1)b_P^++(n+1)(b_Q-b_Q^+)m_Q\rf
\\&=nb_P^++\lf b_P^++((n+1)b_Q-nb_Q^+)m_Q-b_Q^+m_Q\rf
\\&\le nb_P^++\lf b_P^++\{(n+1)b_Q\}m_Q-b_Q^+m_Q\rf
\\&=nb_P^+.
\end{align*}
The last equality holds because 
$$b_P^++\{(n+1)b_Q\}m_Q-b_Q^+m_Q<b_P^++m_Q-b_Q^+m_Q=r_{PQ}\le1.$$
The proof is finished.
\end{proof}

\section{Monoids of effective Weil divisors}\label{sec4}

\begin{defn}
Let $\pi:X\to Z$ be a projective morphism of varieties. Set
$$\Cl(X/Z):=\big\{D\mid D\in\WDiv(X)\big\}/\sim_Z\text{ and } \Eff(X/Z):=\big\{D\ge0\mid D\in\WDiv(X)\big\}/\sim_Z.$$ 
\end{defn}

\begin{lem}\label{lem:0overpt}
Let $f:X\to S$ be a projective morphism between normal quasi-projective varieties, and $D$ an $\Rr$-divisor on $X$. Then $D\ge0$ (resp. $=0$) over a non-empty open susbet of $S$ if and only if $D\ge0$ (resp. $=0$) over a general point of $S$.
\end{lem}
\begin{proof}
We may write $D:=\sum_i d_iD_i$, where $d_i$ are real numbers and $D_i$ are distinct prime divisors on $X$. Possibly shrinking S, we may assume that $D_i$ dominates $S$ for each $i$, and choose an open subset $\emptyset\neq U\subseteq S$ such that $D_i|_{f^{-1}U}$ dominates $U$. By the semi-continuity of the dimension of fibers \cite[\S 2 Excersise 3.22]{GTM52}, possibly shrinking $U$, we may assume that for each point $s\in U$,
\begin{itemize}
  \item $\dim X_s=\dim X-\dim S+\dim s$,
  \item $\dim X_s\cap D_i=\dim X-\dim T+\dim s-1$ for any $i$, and
  \item $\dim X_s\cap D_i\cap D_j=\dim X-\dim S+\dim s-2$ for any $i\neq j$.
\end{itemize}
In particular, $D_i|_{X_s}$ are distinct divisors with no common components.

If $D\ge0$ (resp. $=0$) over an open subset $U'\subseteq S$, then $d_i\ge0$ (resp. $=0$) for any $i$, and thus $D|_{X_s}=\sum_i d_iD_i|_{X_s}\ge0$ (resp. $=0$) for any $s\in U\cap U'$. Conversely, if $D|_{X_s}\ge0$ (resp. $=0$) for some $s\in U$, then by our construction, $D\ge0$ (resp. $=0$) over $U$.
\end{proof}

\begin{lem}\label{lem:trivialovergp}
Let $f:X\to S$ be a projective morphism between normal quasi-projective varieties, $D$ an $\Rr$-Cartier $\Rr$-divisor on $X$, $\eta$ the generic point of $S$ and $X_\eta$ the generic fiber. Suppose that $D|_{X_{\eta}}\sim_{\Rr}0$ (resp. $\sim0$). Then there exists an open set $\emptyset\neq U$ of $S$, such that $D|_{X\times_S U}\sim_{\Rr}0$ (resp. $\sim0$). 
\end{lem}

\begin{proof}
Since $D|_{X_{\eta}}\sim_{\Rr}0$ (resp. $\sim0$), there exist $a_1,\ldots,a_m\in\Rr$ (resp. $\in\Zz$) and $s_1,\ldots,s_m\in \mathcal{K}(X_{\eta})=\mathcal{K}(X)$, such that 
$D|_{X_{\eta}}=(\sum_{i=1}^m a_i(s_i))|_{X_{\eta}}$. Let $D':=D-\sum_{i=1}^m a_i(s_i)$. By Lemma \ref{lem:0overpt}, $D'=0$ over an open subset $U$ of $S$ as $D'|_{X_{\eta}}=0$. Hence $D|_{X\times_S U}\sim_{\Rr}0.$  
\end{proof}

\begin{lem}\label{lem:rcpic}
Assume that $X$ is a smooth rationally connected variety. Then $H^1(X,\Oo_X)=H^2(X,\Oo_X)=0$. Moreover, $\mathrm{Pic}(X)\cong H^2(X,\Zz)$.
\end{lem}

\begin{proof}
By the exponential sequence, we have the following long exact sequence of sheaf cohomology,
$$H^1(X,\Oo_X)\to \mathrm{Pic}(X)\to H^2(X,\Zz)\to H^2(X,\Oo_X).$$
It suffices to show that $H^i(X,\Oo_X)=0$ for $i=1,2$.

By \cite[\S 4 3.3 Propositon]{Kol96}, $X$ is separably rationally connected. Thus by \cite[\S 4 3.8 Corollary]{Kol96},  
$H^0(X,(\Omega_X^1)^{\otimes i})=0$
for every $i>0$. Note that $\mathcal{F}\bigotimes\mathcal{F}=S^2(\mathcal{F})\bigoplus\bigwedge^2(\mathcal{F})$ for any sheaf of $\Oo_X$-modules $\mathcal{F}$, where $S^2(\mathcal{F})$ and $\bigwedge^2(\mathcal{F})$ are the 2nd symmetric product of $\mathcal{F}$ and the 2nd exterior power of $\mathcal{F}$ respectively. It follows that 
$H^0(X,\Omega_X^i)=0$
for $i=1,2$. By Hodge symmetry, $H^i(X,\Oo_X)=0$ for $i=1,2$, and we are done.
\end{proof}

The following theorem might be well-known to experts. For the reader's convenience, we provide a basic proof in details here.

\begin{thm}\label{thm:picisconst}
Let $f:X\to S$ be a smooth projective morphism between smooth quasi-projective varieties such that the fibers $X_s$ is rationally connected for any closed point $s\in S$. Then possibly up to an \'etale base change, $\Pic(X/S)\to \Pic(X_s)$ is an isomorphism for any $s\in S$.
\end{thm}

\begin{proof}
\noindent\textbf{Step 1}. 
In this step, we construct an open subset $\emptyset\neq V\subset S$ such that 
$$\phi_s: \Pic(X/S)\to \Pic\left(X_s\right)$$
is an injective between abelian groups for any $s\in V$.

First, we claim that
$$\phi_\eta: \Pic(X/S)\to \Pic\left(X_\eta\right)$$
is an isomorphism between abelian groups, where $\eta$ be the generic point of $S$ and $X_\eta$ is the generic fiber. In fact, the surjectiveness is easily implied by taking the closures of divisors on $X_\eta$. For the injectiveness, if $L|_{X_\eta}\sim 0$, then $L|_{X_U}\sim0$ for some open subset $U\subseteq S$ by Lemma \ref{lem:trivialovergp}, where $X_U:=X\times_SU$. Therefore $L\sim \sum_i n_iD_i$ for some integers $n_i$ and prime divisors $D_i$ that are vertical over $S$. Since $f:X\to S$ is smooth, projective and geometrically connected, we must have $D_i=f^*f(D_i)$ for each $i$ and thus $L\sim \sum_i n_if^*(f(D_i))=0\in\Pic(X/S)$. This finishes the proof of the claim.

In particular, since $\Pic^0(X_\eta)=0$ by the fact that $\mathrm{H}^1(X_\eta,\Oo_{X_\eta})=0$, we know that $$\Pic(X/S)=\Pic\left(X_\eta\right)=\mathrm{NS}\left(X_\eta\right)$$ 
is a finite generated abelian group, hence with finitely many non-trivial torsions $L_1,\dots,L_m$. Note that since $\phi_\eta$ is an isomorphism, $L_i|_{X_\eta}$ is a non-trivial torsion and thus $\mathrm{H}^0(X_\eta,L_i|_{X_\eta})=0$ for each $1\le i\le m$. By the upper-semicontinuity, there exists an open subset $V\subseteq S$ such that $\mathrm{H}^0(X_s,L_i|_{X_s})=0$ for each $1\le i\le m$ and any $s\in V$. We show that
$$\phi_s: \Pic(X/S)\to \Pic(X_s)$$
is injective for any $s\in V$. Indeed, if $L|_{X_s}\sim 0$ for some non-trivial $L\in\Pic(X/S)$, then we can see that $L|_{X_\eta}\equiv0$ in $\mathrm{NS}(X_\eta)_\Qq$. This is equivalent to say that $L|_{X_\eta}$ is a torsion in $\Pic(X_\eta)$ and hence $L$ is a torsion in $\Pic(X/S)$. Then $L=L_i$ for some $1\le i\le m$ which is impossible since $L_i|_{X_s}$ is not trivial for any $s\in V$. We emphasize here that $V$ only depends on the torsions of $\Pic(X_\eta)$.

\medskip

\noindent\textbf{Step 2}. 
In this step, we show that there exists an \'etale morphism $S'\to S$ with $S'$ irreducible such that for any \'etale morphism $S''\to S'$ with $S''$ irreducible, the map 
$$\Pic\left(X_{S''}:=X\times_S S''/S''\right)\to \Pic\left(X_{s''}:=X\times_{S}\Spec~ k(s'')\right)$$ 
is injective for any $s''\in S''$.

For any finite field extensions $K''/K'/k(\eta)$, there exist \'etale morphisms $S''\to S'\to S$ with $S',S''$ irreducible such that the induced $k(\eta'')/k(\eta')/k(\eta)$ is exactly $K''/K'/k(\eta)$, where $\eta'$ and $\eta''$ are the generic points of $S'$ and $S''$ respectively. Moreover, the induced group homomorphisms
$$\Pic\left(X_{\eta'}\right)\to\Pic\left(X_{\eta''}\right)\text{ and }\Pic\left(X_{\eta'}\right)\to\Pic\left(X_{\overline{k(\eta)}}\right)$$
are injective by computing the cohomology under flat base change (consider $\mathrm{H}^0$ of $L$ and $L^{-1}$). 
Since all the groups above can be regarded as submodules of the finitely generated $\Zz$-module $\Pic(X_{\overline{k(\eta)}})=\mathrm{NS}(X_{\overline{k(\eta)}})$, we know there exists $k(\eta')$ such that $\Pic(X_{\eta'})\to\Pic(X_{\eta''})$ is an isomorphism for any field extension $k(\eta'')/k(\eta')$ defined above. In particular, any torsion in $\Pic(X_{\eta''})$ is linearly equivalent to a pullback of some torsion in $\Pic(X_{\eta'})$. Combined with the arguments in {\bf Step 1}, replacing $S'$ with an open subset $V'$ of $S'$, we may finish this step.

\medskip

In what follows, we may replace $X\to S$ with $X_{S'}\to S'$ as constructed in {\bf Step 2} but keep using the same symbols.

\noindent\textbf{Step 3}. In this step, we claim that 
$$\phi_s: \Pic(X/S)\to\Pic(X_{s})$$ 
is isomorphic for any very general closed point $s\in S$. Indeed, we only need to show that $\phi_s$ is surjective.

Consider the relative Hilbert functor of $X/S$, which has countably many irreducible components $H_i$ and $h_i: H_i\to S$ are all projective. For each $h_i$ with $h_i(H_i)=S$, we let $\tilde{h}_i: (H_i)_{\mathrm{red}}\to S$ be the reduction and
$$U_i:=\left\{s\in S~\big|~\text{$\tilde{h}_i^{-1}(s)$ as a scheme is reduced}\right\}.$$
Then $U_i$ is an open subset of $S$. For each $h_i$ such that $h_i(H_i)\subset S$ is a proper closed subvariety of $S$, we define $Z_i:=h_i(H_i)$. We may set
$$\hat{V}:=S~\backslash\bigcup_{h_i(H_i)\neq S}Z_i\text{ and }\hat{U}=\hat{V}\bigcap_{h_i(H_i)=S}U_i.$$
It is clear that $\hat{U}$ is a very general subset of $S$. 

For any closed point $s\in\hat{U}$ and any prime divisor $D_s$ on $X_s$, we can regard $D_s$ as a morphism $\Spec~k(s)=\Spec~\Cc\to H_i$ over $S$ for some $i$, or equivalently, a closed point on the fiber $H_{i,s}:=H_i\times_S\{s\}$ for some $i$. Then any other closed point in the same connected component of $H_{i,s}$ represents the same class $[D_s]\in\Pic(X_s)$ since $\Pic^0(X_s)=0$. Therefore we can assume that $D_s$ corresponds to a smooth point $y$ of $(H_{i,s})_{\mathrm{red}}$. Notice that $s\in U_i$, so $(H_{i,s})_{\mathrm{red}}=(H_i)_{\mathrm{red},s}$ and then we will have an \'etale section of $\tilde{h}_i$ that extends $s\mapsto y$.
In other words, there exists an \'etale neighborhood $s'\in S'$ of $s\in S$ such that $H_{i,S'}\to S'$ admits a section $S'\to H_{i,S'}$, which corresponds to a flat family $D'\to S'$ with $[D'_{s'}]=[D_s]\in\Pic(X_s)=\Pic(X_{s'})$. By considering the generators of $\Pic(X_s)$ and writing them as difference of two effective divisors we can assume that this $S'$ works for all these effective divisors. It follows that $\Pic(X_{S'}/S')\to\Pic(X_{s})$ is surjective. 

Since $\Pic(X/S)\to\Pic(X_{S'}/S')$ is an isomorphism by the argument in {\bf Step 2}, we know that $\Pic(X/S)\to\Pic(X_s)$ is isomorphic for any closed $s\in \hat{U}$ as well.

\medskip

\noindent\textbf{Step 4}. In this step, we finish the proof.

For any closed point $s\in S$, by Lemma \ref{lem:rcpic}, we know that $\mathrm{H}^2(X_s,\Zz)=\Pic(X_s)$. Thus there exists a contractible analytic neighborhood $s\in U_s$ such that $X_{U_s}:=X\times_SU_s\to U_s$ is a (topologically) trivial fibration and $\mathrm{H}^2(X_{s'},\Zz)$ is canonically identical to $\mathrm{H}^2(X_{U_s},\Zz)$ for any $s'\in U_s$. Take some point $s'\in U_s\cap\hat{U}$, then for any divisor $D_s$ on $X_s$, $[D_s]\in\Pic(X_s)$ corresponds to an element $L_{s'}$ in $\Pic(X_{s'})=\mathrm{H}^2(X_{s'},\Zz)=\mathrm{H}^2(X_{s},\Zz)$. Since $\phi_{s'}:\Pic(X/S)\to\Pic(X_{s'})$ is an isomorphism (by \textbf{Step 3}), we can find a line bundle $L$ on $X$ such that $L|_{X_{s'}}\sim L_{s'}$. Then $L|_{U_s}$ must correspond to the element $[D_s]\in\mathrm{H}^2(X_{U_s},\Zz)$ and hence we have $L|_{X_s}=[D_s]\in\Pic(X_s)$. Therefore $\phi_s:\Pic(X/S)\to\Pic(X_s)$ is surjective, and thus isomorphic.

If $s\in S$ is point with $\dim\bar{s}>0$, then by the statements in {\bf Step 1} we know for general points $s'\in\bar{s}$, we have the injection
$$
\Pic(X_s)\to\Pic(X_{s'})
$$
and by {\bf Step 2} we know that $\Pic(X/S)\to\Pic(X_{s})$ is also injective, then the statement follows since the composition $\Pic(X/S)\to\Pic(X_{s'})$ is bijective.
\end{proof}

\begin{thm}[cf.\ {\cite[Theorem 5]{Sho20}}]\label{thm:effgpisfg}
Let $X$ be a normal projective variety that is of Fano type. Then the commutative monoid $\Eff(X)$ is finitely generated.
\end{thm}

\begin{proof}
Possibly replacing $X$ with a small $\Qq$-factorization, we may assume that $X$ is $\Qq$-factorial. 

Since $X$ is of Fano type, there exists a klt pair $(X,B)$ with a $\Qq$-boundary $B$, such that $-(K_X+B)$ is ample. By \cite[Theorem 3.1]{Totaro12}, there exist prime divisors $D_1,\ldots,D_r\ge0$ which generate $\Cl(X)$. We may choose a positive integer $m$ and an ample $\Qq$-divisor $A\sim_{\Qq}-(K_X+B)$ such that 
\begin{itemize}
  \item $mA$ is an integral divisor,
  \item $mA\sim-m(K_X+B)$,
  \item $A-\frac{1}{m}\sum_{i=1}^r D_i\ge0$, and
  \item $(X,B+\frac{1}{m}\sum_{i=1}^r D_i+A)$ is klt.
\end{itemize}
Let $\Delta_j$ ($j=1,2,\ldots,2r$) be all the possible divisors of the form $B\pm \frac{1}{m}D_i+A$, where $1\le i\le r$. Since $(X,\Delta_j)$ is klt for each $j$, by \cite[Corollary 1.1.9]{BCHM10}, the Cox ring
$$\bigoplus_{(m_1,\ldots,m_{2r})\in\Zz_{\ge0}^{2r}}H^0\left(X,\Oo_X\left(\bigg\lfloor \sum_{j=1}^{2r} m_j\left(K_X+\Delta_j\right)\bigg\rfloor\right)\right)$$
is finitely generated. Notice that $m(K_X+\Delta_j)\sim D_j$ for any $1\le j\le 2r$, where $D_{j}:=-D_{j-r}$ for $r+1\le j\le 2r$. Therefore by \cite[Proposition 7.8]{AM}, we know the subring

$$\bigoplus_{(m_1,\ldots,m_{2r})\in\Zz_{\ge0}^{2r}}H^0\left(X,\sum_{j=1}^{2r} m_jD_j\right)$$
is also finitely generated. 
This immediately implies that $\Eff(X)$ is finitely generated.
\end{proof}

\begin{proof}[Proof of Theorem \ref{thm:effisconst}]
By assumption, there is an effective $\Qq$-divisir $\Delta$ such that $(X,\Delta)$ is klt and $-(K_X+\Delta)$ is ample over $S$. Let $g:Y\to X$ be a log resolution of $(X,\Delta)$, possibly passing to an \'etale base change of $S$, we can assume that
\begin{itemize}
    \item $f:(X,\Delta)\to S$ is flat and $X_s$ is normal for any $s\in S$,
    \item $(Y,E)\to S$ is a log smooth family, where $E$ is the reduced $g$-exceptional divisors, and
    \item $E|_{Y_s}$ is the exceptional locus of $g_s: Y_s\to X_s$, and $E_{i,s}:=E_i|_{Y_s}$ is irreducible for any irreducible component $E_i$ of $E$.
\end{itemize}
Since $H^i(Y_s,\Oo_{Y_s})=H^i(X_s,\Oo_{X_s})=0$ for $i=1,2$, by Theorem \ref{thm:picisconst}, up to another \'etale base change of $S$, we may further assume that 
$$\phi^Y_s:\Pic(Y_\eta)\cong\Pic(Y/S)\to\Pic(Y_s)$$
is an isomorphism for any $s\in S$. Notice that 
$$\mathrm{Cl}(X/S)=\Pic(Y/S)/\oplus_i\Zz[E_i]\text{ and }\mathrm{Cl}(X_s)=\Pic(Y_s)/\oplus_i\Zz[E_{i,s}],$$
therefore $\phi^Y_s$ actually induces an isomorphism $\phi_s:\mathrm{Cl}(X/S)\to\mathrm{Cl}(X_s)$ for any $s\in S$. 

We may take $D_1,...,D_r$ to be generators of $\mathrm{Cl}(X/S)$. Let $A\sim_S -d(K_X+\Delta)$ be a relatively very ample Cartier divisor over $S$ such that there exists effective Weil divisors $B_1,...,B_{2r}$ satisfying $B_i\sim_S A+D_i$ and $B_{r+i}\sim_S A-D_i$. We may take $h:W\to X$ to be a log resolution of $(X,\sum B_i+\Delta)$ such that up to
shrinking $S$ the following holds.
\begin{enumerate}
    \item $h$ factors through $g$.
    \item $(W,\Delta_W+F+\sum B_{i,W}+A_W)\to S$ is a log smooth family, where $\Delta_W=h_*^{-1}\Delta,~B_{i,W}=h_*^{-1}B_i,~0\le A_W\sim_S h^*A$ and $F$ is the reduced exceptional divisors.
    \item $A_W$ have no common component with $\Delta_W+F+\sum B_{i,W}$.
    \item $F|_{W_s}$ is the exceptional locus of $W_s\to X_s$.
\end{enumerate}

Now let $L$ be any Weil divisor on $X$, then $L=\sum_{i=1}^ra_iD_i\in \mathrm{Cl}(X/S)$ for some $a_i\in\Zz$ and
\begin{align*}
L+\sum_i|a_i|A&\sim_S\sum_{a_i>0}a_iB_i-\sum_{a_i<0}a_iB_{i+r}=\sum_{a_i>0}a_iB_i+\sum_{a_i<0}(-a_i)B_{i+r}\ge0.
\end{align*}
Let
$$L_W:=\sum_{a_i>0}a_iB_{i,W}+\sum_{a_i<0}(-a_i)B_{i+r,W}-\sum_i|a_i|f^*A,$$
then we have $h_*L_W\sim_S L$ and $h_*\Oo_W(L_W+nF)=\Oo_X(h_*L_W)$ for any $n\gg 0$.

Notice that
$$L\sim_S md(K_X+\Delta)+\sum_{a_i>0}a_iB_i+\sum_{a_i<0}(-a_i)B_{i+r}+\left(m-\sum_i|a_i|\right)A$$
for any $m$. Let $K_W+\Delta'=h^*(K_X+\Delta)$ and write $\Delta'=\Delta'_{1}-\Delta'_{2}$, where $\Delta'_{1}\ge0,\Delta'_{2}\ge0$ and $\Delta'_{1}\cap\Delta'_{2}=0.$ For any $m\gg n$, we have

\begin{align*}
L_W+nF+md\Delta'_{2}\sim_S& ~md\left(K_W+\Delta'_{1}\right)+nF+\sum_{a_i>0}a_iB_{i,W}\\&+\sum_{a_i<0}(-a_i)B_{i+r,W}+\left(m-\sum_i|a_i|\right)A_W.
\end{align*}
Notice that $md\Delta'_{2}$ is $h$-exceptional, therefore we also have 
$$h_*\Oo_W\left(L_W+nF+md\Delta'_{2}\right)=\Oo_X(h_*L_W)$$
for any $m, n\gg0$ such that $m-n\gg0$. Since $D_W$ is general, $D_W+\sum_iB_{i,W}+F+\Delta_W$ is simple normal crossing. Let 
$$\Delta'':=\Delta'_{1}+\frac{1}{md}\left(nF+\sum_{a_i>0}a_iB_{i,W}+\sum_{a_i<0}(-a_i)B_{i+r,W}+\left(m-\sum_i|a_i|\right)A_W\right),$$
then $(W,\Delta'')$ is an snc klt pair over $S$ such that 
$$md\left(K_W+\Delta''\right)\sim_S L_W+nF+md\Delta'_{2}.$$
Hence by \cite[Theorem 1.8]{HMX13} we have $r:=h^0(W_s,md(K_{W_s}+\Delta''_s))$ is independent of $s\in S$. By Grauert's theorem \cite[Corollary 12.9]{GTM52} we have 
$$f_*h_*\Oo_W\left(L_W+nF+md\Delta'_{2}\right)\otimes k(s)\to H^0\left(W_s,\left(L_W+nF+md\Delta'_{2}\right)|_{W_s}\right)$$
is an isomorphism for any $s\in S$ and $m\gg n\gg0$. Since $F|_{W_s}$ is the exceptional locus of $h_s: W_s\to X_s$ and $md\Delta'_{2}|_{W_s}$ is exceptional, we have 
$$H^0\left(W_s,\left(L_W+nF+md\Delta'_{2}\right)|_{W_s}\right)=H^0\left(X_s,(h_s)_{*}\left(L_W|_{W_s}\right)\right)$$
and it follows that
$$f_*\Oo_X(L)\otimes k(s)=f^*\Oo_X(h_*L_W)\otimes k(s)\to H^0\left(X_s,(h_s)_{*}\left(L_W|_{W_s}\right)\right)$$
is isomorphic and therefore $f_*\Oo_X(L)$ is locally free of rank $r$. 

If we let $h=g\circ\alpha$, then since $g_*(\alpha_*L_W-g^{-1}_*L)=h_*L_W-L\sim_S 0$ we have 
\[\alpha_*L_W-g^{-1}_*L\sim_S(\alpha_*L_W-g^{-1}_*L)-g^*(h_*L_W-L)\]
and the latter is a Weil divisor supported on $\Exc(g)=\Supp E$. Therefore $\alpha_*L_W$ and $g^{-1}_*L$ have the same image under $g_*:\Pic(Y/S)\to\Cl(X/U)$. Notice that $\phi_s^Y(\alpha_*L_W)=(\alpha_{s})_*(L_{W}|_{W_s})$ and $(g_s)_*((\alpha_{s})_*(L_{W}|_{W_s}))=(h_s)_*(L_W|_{W_s})$, then we have 
$$
(h_s)_*(L_W|_{W_s})=(g_s)_*\phi_s^Y(\alpha_*L_W)=(g_s)_*\phi^Y_s(g^{-1}_*L)=\phi_s(L)\in\Cl(X_s).
$$ Hence we can rewrite the above isomorphism as
$$f_*\Oo_X(L)\otimes k(s)\to H^0(X_s,\phi_s(L)),$$
and this implies that $\phi_s$ actually induces an isomorphism
$$\phi_s|_{\Eff(X/S)}:\Eff(X/S)\to\Eff(X_s)$$
for any $s\in S$. Indeed this is because 
$$
L\in\Eff(X/S)\Leftrightarrow f_*\Oo_X(L)\neq0\Leftrightarrow r>0\Leftrightarrow \phi_s(L)\in\Eff(X_s).
$$ 
The proof is finished.
\end{proof}

\section{Boundedness of complements for exceptional generalized pairs}\label{sec5}
In this section, we show the boundedness of complements for exceptional generalized pairs which is the first step towards our main result, cf. {\cite[Theorem 7]{Sho20}}.
\begin{thm}\label{thm:exc}
Let $d,p,I$ be positive integers, $\fR\subseteq[0,1]\cap\Qq$ a finite set, and $\Phi:=\Phi(\fR)$. Then there exists a finite set $\cN$ of positive integers depending only on $d, p,I,$ and $\Phi$, satisfying $I\mid\cN$ and the following.

Assume that $(X,B+\bM)$  is a projective g-pair of dimension $d$ such that $X$ is of Fano type, $p\bM$ is b-Cartier, $(X,B_\Phi+\bM)$ is exceptional, and $(X,B+\bM)$ is $\Rr$-complementary. Then $(X,B+\bM)$ is $\cN$-complementary.
\end{thm}

\subsection{Monoids of effective divisors for exceptional generalized pairs}

\begin{defn}\label{defn:exc}
We say that a projective g-pair $(X,B+\bM)$ is \emph{exceptional} if it is $\Rr$-complementary and any $\Rr$-complement of $(X,B+\bM)$ is klt.
\end{defn}

\begin{rem}
We remark that $(X,B+\bM)$ is exceptional if and only if it is $\Rr$-complementary and $(X,B+G+\bM)$ is klt for any effective $\Rr$-divisor $G\sim_{\Rr}-(K_X+B+\bM_X)$. Indeed, otherwise, there exists a non-lc g-pair $(X,B'+\bM)$ such that $K_X+B'+\bM_X\sim_\Rr0$ and $B'\ge B$. Let $(X,B''+\bM)$ be an $\Rr$-complement of $(X,B+\bM)$. Take a positive real number $\delta$ such that $(X,(\delta B'+(1-\delta)B'')+\bM)$ is lc but not klt. In particular, $(X,(\delta B'+(1-\delta)B'')+\bM)$ is a non-klt $\Rr$-complement of $(X,B+\bM)$, a contradiction.
\end{rem}

\begin{lem}\label{lem:excbir}
Let $(X,B+\bM)$ be a projective g-pair that is exceptional and $f:Y\to X$ a birational morphism. Assume that $(Y,B_Y+\bM)$ is an $\Rr$-complementary g-pair such that $f_*B_Y\ge B$. Then $(Y,B_Y+\bM)$ is exceptional.
\end{lem}
\begin{proof}
Let $(Y,B_Y^++\bM)$ be an $\Rr$-complement of $(Y,B_Y+\bM)$ and let $B^+:=f_*B_Y^+$. Then $B^+\ge B$ and
$$K_Y+B_Y^++\bM_{Y}=f^*\left(K_X+B^++\bM_X\right).$$
Since $(X,B+\bM)$ is exceptional, $(X,B^++\bM)$ is klt which implies that $(Y,B_Y^++\bM)$ is klt. Therefore $(Y,B_Y+\bM)$ is exceptional.
\end{proof}

\begin{prop}\label{prop:excgpairbdd}
Let $d,p$ be positive integers and $\fR\subseteq[0,1]\cap\Qq$ a finite set. Then there exists a positive integer $m$ depending only on $d,p$, and $\fR$ satisfying the following.

Assume that $(X,B+\bM)$ is a $\Qq$-factorial projective g-pair such that
\begin{enumerate}
  \item $\dim X=d$,
  \item $B\in\fR$, $p\bM$ is b-Cartier,
  \item $X$ is of Fano type, and
  \item $(X,B+\bM)$ is exceptional.
\end{enumerate}
Then $\Eff(X)$ has at most $m$ prime generators.
\end{prop}
\begin{proof}
According to \cite[Theorem 1.10]{Bir19} (see also \cite[Theorem 1.1]{Chen23}), there exists a positive integer $n'$ which only depends on $d,p$ and $\fR$ such that $(X,B+\bM)$ has a monotonic $n'$-complement $(X,B'+\bM)$. As $(X,B+\bM)$ is exceptional, $(X,B'+\bM)$ is klt. Moreover, since $n'(K_X+B'+\bM_X)$ is Cartier, we know that $(X,B'+\bM)$ is $\frac{1}{n'}$-lc. In particular, $(X,B+\bM)$ has an $\Rr$-complement that is $\frac{1}{n'}$-lc and thus so does $(X,B)$ by Lemma \ref{lem:Rcompgpairimplypair}. By \cite[Theorem 1.4(3)]{CX22}, there exists a positive integer $n$ which only depends on $d,p,\fR$ and $n'$ such that $(X,B)$ has a $n$-complement $(X,B^+)$ that is $\frac{1}{n'}$-lc. By \cite[Theorem 1.11]{Bir19}, $(X,B^+)$ belongs to a bounded family $(\cX,\cB^+)$ over $T$. 

Note that by our construction there exists a dense subset of points $\{t_i\}_{i\in I}\subseteq T$ such that $(\cX_{t_i},\cB_{t_i}^+:=\cB^+|_{\cX_{t_i}})$ is a $\Qq$-factorial $\frac{1}{n'}$-lc pair,  $K_{\cX_{t_i}}+\cB^+_{t_i}\sim_\Qq0$ and $\cB^+_{t_i}$ is big for any $i\in I$. Then by Proposition \ref{prop:FTovergp} we may assume that possibly after an \'etale base change, for each component $T_i$ of $T$, $\cX_i:=\cX\times_TT_i$ is of Fano type over the generic point of $T_i$. Then the proposition follows from Theorems \ref{thm:effisconst} and \ref{thm:effgpisfg}.
\end{proof}

The following proposition is essentially contained in the proof of \cite[Proposition 2.9]{HX15}. For the reader's convenience, we give a proof here.
\begin{prop}\label{prop:FTovergp}
Let $(\cX,\cB)$ be a pair, $\cX\to T$ a contraction between normal quasi-projective varieties. Assume that $\cB$ is a $\Qq$-divisor. Let $\epsilon$ be a positive real number and $\{t_i\}_{i\in I}\subseteq T$ a dense subset of points such that $(\cX_{t_i},\cB_{t_i}:=\cB|_{\cX_{t_i}})$ is $\Qq$-factorial $\epsilon$-lc,  $K_{\cX_{t_i}}+\cB_{t_i}\sim_\Qq0$ and $\cB_{t_i}$ is big for any $i\in I$. Then there exists an \'etale morphism $S\to T$, such that $\cX\times_TS$ is of Fano type over $S$.
\end{prop}
\begin{proof}
We claim that there exists a positive rational number $a$ which only depends on $(\cX/T,\cB)$ such that $(\cX_{t_i},(1+a)\cB_{t_i})$ is $\frac{\epsilon}{2}$-lc for any $i\in I$. By \cite[Definition-Lemma 2.8]{Xu20}, possibly stratifying $T$ into a disjoint union of finitely many constructible subsets and taking \'etale coverings, we may assume that there exists a decomposition $T=\bigsqcup T_\alpha$ into irreducible smooth strata $T_\alpha$ such that for each $\alpha$, $(\cX\times_T T_\alpha,\cB\times T_\alpha)$ admits a fiberwise log resolution.
We may assume that $(\cX,\cB)$ admits a fiberwise log resolution $\mu:\cX'\to\cX$ and write
$$K_{\cX'} +\cB':=\mu^*\left(K_\cX+\cB\right),$$
for some $\Qq$-divisor $\cB'$ on $\cX'$. By assumption, the coefficients of $\cB'|_{\cX_{t_i}'}$ are $\le1-\epsilon$. We finish the claim by showing that there exists a positive integer $M_0$ depending only on $\mu:\cX'\to\cX$ and $\cB$, such that $\mult_{F_i}\mu_{i}^*\cB_{t_i}\le M_0$ for any $i\in I$ and any prime divisor $F_i$ on $\cX_{t_i}'$, where $\mu_{i}$ denotes the morphism $\cX_{t_i}'\to\cX_{t_i}$. Take a Cartier divisor $\mathcal{H}\ge\cB$ on $\cX$ and let $\mathcal{H}':=\mu^*\mathcal{H}$. As 
$$\mult_{F_i}\mu_{i}^*(\mathcal{H}|_{\cX_{t_i}})=\mult_{F_i}\mathcal{H}'|_{\cX_{t_i}'}$$
for any $i\in I$ and any prime divisor $F_i$ on $\cX_{t_i}'$,
it is clear that $\mult_{F_i}\mu_{i}^*(\mathcal{H}|_{\cX_{t_i}})\le M_0$ for some positive integer $M_0$ which only depends on the divisor $\mathcal{H}'$. In particular,
$$\mult_{F_i}\mu_{i}^*\cB_{t_i}\le \mult_{F_i}\mu_{i}^*(\mathcal{H}|_{\cX_{t_i}})\le M_0.$$

\smallskip

Now by \cite[Proposition 2.4]{HX15} possibly shrinking $T$, we may assume that $K_{\cX}+(1+a)\cB$ is $\Qq$-Cartier and klt. We may write
$$K_{\cX'} +\cD_a':=\mu^*\left(K_\cX+\left(1+a\right)\cB\right)+\cE_a,$$
where $\cD_a'$ and $\cE_a$ are effective $\Qq$-divisors with no common components. Since $\cE_a|_{\cX_t'}$ is effective and exceptional over $\cX_t$, when $m$ is sufficiently divisible, we see that
\begin{align*}
&\ \ \ \ h^0\left(\cX_t,m\left(K_\cX+\left(1+a\right)\cB\right)|_{\cX_t}\right)\\&
=h^0\left(\cX_t',\mu_t^*m\left(K_\cX+\left(1+a\right)\cB\right)|_{\cX_t}+\cE_a|_{\cX_t'}\right)
\\&= h^0\left({\cX_t'},m\left(K_{\cX'} + \cD_a'\right)|_{\cX_t'}\right),
\end{align*}
where $\mu_t$ denotes the morphism $\cX_t'\to\cX_t$. Hence $$\vol\left(\cX_t,\left(K_\cX+\left(1+a\right)\cB\right)|_{\cX_t}\right)=\vol\left(\cX'_t,\left(K_{\cX'}+\cD_a'\right)|_{\cX'_t}\right).$$
By \cite[Theorem 1.8]{HMX13}, $\vol(\cX_t,(K_\cX+(1+a)\cB)|_{\cX_t})$ is locally constant for $t\in T.$ Therefore $K_{\cX_\eta}+(1+a)\cB_\eta$ is big, where $\eta$ is the generic point of $T$. 

\smallskip

According to \cite[Theorem 1.8]{HMX13} again, we have $\kappa_{\sigma}(K_{\cX_\eta}+\cB_\eta)=0$. Hence by \cite[Theorem 1.3]{Gongyo11}, we see that $\kappa(K_{\cX_\eta}+\cB_\eta)=0$. In particular, we can find an effective $\Qq$-divisor $E|_{\cX_\eta}$ on $\cX_\eta$ such that $E\sim_{\Qq}K_{\cX_\eta}+\cB_\eta$. By \cite[Lemma 3.2.1]{BCHM10}, there exists an effective $\Qq$-divisor $E$ on $\cX$ such that $E\sim_{\Qq,U}K_{\cX}+\cB$ and $E|_{\cX_\eta}=E_\eta$. Since $E|_{\cX_{t_i}}\sim_\Qq K_{\cX_{t_i}}+\cB_{t_i} \sim_\Qq0$ and $\{t_i\}_{i\in I}$ is a dense subset of $T$, $E|_{\cX_t}\sim_\Qq0$ over a general fiber $t$. As $E|_{\cX_t}\ge0$, one can see that $E|_{\cX_t}=0$ and thus $E=0$ over an open subset of $T$ by Lemma \ref{lem:0overpt}. Therefore possibly shrinking $T$, we may assume that $K_{\cX}+\cB\sim_{\Qq,T}0$, and $\cB$ is big over $T$. The proof is finished.
\end{proof}

\subsection{Proof of Theorem \ref{thm:exc}}

\begin{prop}\label{prop:compbdd}
Let $d,p,m,$ and $I$ be positive integers. Then there exists a finite set $\cN$ of positive integers depending only on $d, p,m,$ and $I,$ satisfying $I\mid\cN$ and the following.

Assume that $(X/Z\ni z,B+\bM)$ is a projective g-pair such that
    \begin{enumerate}
      \item $\dim X=d$,
      \item $X$ is of Fano type over $Z$,
      \item $(X/Z\ni z,B+\bM)$ is $\Rr$-complementary, and
      \item the number of components of $B$ is at most $m$, and $p\bM$ is b-Cartier.
    \end{enumerate}
    Then $(X/Z\ni z,B+\bM)$ is $\cN$-complementary.
\end{prop}

\begin{proof}
Suppose on the contrary that there exist a sequence of g-pairs $(X_i/Z_i\ni z_i,B_{i}+\bM_{i})$ satisfying the conditions, such that
\[
n_i:=\min\big\{n \mid (X_i/Z_i\ni z_i,B_{i}+\bM_i)\text{ is }n\text{-complementary},\text{ and }I\mid n
\big\}
\]
is a strictly increasing sequence of positive integers. For each $i$, we may write $B_i:=\sum_{j=1}^mb_{ij}B_{ij}$, where $B_{i1},\dots,B_{im}$ are prime divisors. Since $b_{ij}\in[0,1],$ possibly passing to a subsequence, we may assume that either $\{b_{ij}\}_{i=1}^{\infty}$ is a decreasing sequence or an increasing sequence, and $\lim_{i\to\infty}b_{ij}=b_j$ for some real number $b_j\in[0,1]$ for any $1\le j\le m$. 

Let $b_{ij}':=\min\{b_{ij},b_j\}$ for any $i,j$, and 
$$B_{i}':=\sum_{j=1}^mb_{ij}'B_{ij}\text{ for any }i.$$ 
Note that the set $\{b_{ij}'\}_{i\ge1,m\ge j\ge1}$ satisfies the DCC and $(X_i/Z_i\ni z_i, B_{i}'+\bM_i)$ is $\Rr$-complementary for each $i\ge1$. By \cite[Theorem 1.1]{Chen23}, there exists a positive integer $n'$ which only depends on $d,p,I,$ and $\{b_{ij}'\}_{i,j}$ satisfying $I\mid n'$, such that $(X_i/Z_i\ni z_i,B_{i}'+\bM_i)$ is $n'$-complementary for any $i$. 

There exists a positive integer $N_0$ such that 
$$b_{ij}-b_j<\frac{1-\{(n'+1)b_j\}}{n'+1}$$
for any $i\ge N_0$ and any $1\le j\le m,$ as $\lim_{i\to \infty}b_{ij}=b_j$ for any $1\le j\le m$. 
Since either $b_{ij}'=b_{ij}$ or $1\ge b_{ij}>b_j=b_{ij}'$ and 
$$(n'+1)b_{ij}'<(n'+1)b_{ij}<1+\lfloor (n'+1)b_{ij}'\rfloor,$$ 
we have
$$\lfloor(n'+1)\{b_{ij}\}\rfloor=\lfloor(n'+1)\{b_{ij}'\}\rfloor\text{ and }  n'\lfloor b_{ij}\rfloor=n'\lfloor b_{ij}'\rfloor$$ 
for any $1\le j\le s$ and any $i\ge N_0$. It follows that $(X_i/Z_i\ni z_i,B_{i}+\bM_i)$ is $n'$-complementary for any $i\ge N_0$, a contradiction.
\end{proof}

\begin{proof}[Proof of Theorem \ref{thm:exc}]
By Lemma \ref{lem:excbir}, we can replace $X$ with a small $\Qq$-factorialization and thus assume that $X$ is $\Qq$-factorial. According to \cite[Lemma 7.2]{Bir19}, the g-pair $(X,B_\Phi+\bM)$ is $\delta$-lc for some $\delta>0$ depending only on $d,p$, and $\Phi$. It follows that $B_{\Phi}\in\Phi_0:=\Phi\cap[0,1-\delta]$ which is a finite set of rational numbers. Moreover, by \cite[Theorem 1.1]{Bir19}, $(X,B_\Phi)$ belongs to a log bounded family. Possibly replacing $I$ by a multiple, we may assume that $p|I$ and $I(\fR\cup\Phi_0)\subseteq\Zz$. 

By Proposition \ref{prop:excgpairbdd}, there exists a positive integer $m$ depending only on $d,p$, and $\Phi$ such that we can find finitely many prime divisors $D_1,\dots,D_m$ (containing all the components of $B_{\Phi}$) satisfying that if we let 
$$\fD:=\Zz(\{D_1,\dots,D_m\})\text{ and }\fD^+:=\Zz_{\ge0}(\{D_1,\dots,D_m\}),$$
then $B_\Phi\in\fD_\Qq:=\fD\otimes\Qq$ and $D\in\fD^+/\sim$ for any effective Weil divisor $D$. By Proposition \ref{prop:compbdd}, there exists a finite set $\cN\subseteq\Zz_{>0}$ depending only on $d,p, m,$ and $I$ satisfies $I\mid \cN$, such that if we additionally require $B\in\fD_\Rr:=\fD\otimes\Rr$, then $(X,B+\bM)$ is $\cN$-complementary. In what follows, we will show that $\cN$ has the required property.

We now first define a homomorphism $\mu:\WDiv(X)_\Rr\to\fD_\Rr$. Let 
$$\mu(D_i):=D_i\text{ for any }1\le i\le m,$$
and for any prime divisor $P\notin\{D_1,\dots,D_m\}$, we let
$\mu(P):=\tilde{P},$
where $\tilde{P}\in\fD^+$ such that $P\sim\tilde{P}$. For any $\Rr$-divisor $\sum a_iP_i$, we define
$$\mu\left(\sum a_iP_i\right)=\sum a_i\mu(P_i).$$ 
It is easy to see that $\mu$ is a homomorphism. It is clear that $\mu|_{\fD_\Rr}={\rm id}$, $\mu(P)\sim P$ for any divisor $P$, and $\mu(Q')\ge \mu(Q)$ for any $\Rr$-divisors $Q'\ge Q$. 

Suppose that $(X,B+G+\bM)$ is an $\Rr$-complement of $(X,B+\bM)$ for some $\Rr$-divisor $G\ge0$. Since $B+G\ge B\ge B_\Phi$ and $B_\Phi\in\fD_\Qq$, we have 
$$\mu(B+G)\ge\mu(B)\ge\mu(B_\Phi)=B_\Phi$$ 
and 
$$K_X+\mu(B+G)+\bM_X\sim_\Rr K_X+B+G+\bM_X\sim_\Rr0.$$
By our assumption that $(X,B_\Phi+\bM)$ is exceptional, we infer that $(X,\mu(B+G)+\bM)$ is a klt $\Rr$-complement of itself. In particular, $(X,\mu(B)+\bM)$ is $\Rr$-complementary. It implies that $(X,\mu(B)+\bM)$ is $n$-complementary for some $n\in\cN$. More precisely, there exists an effective Weil divisor $E$ such that
$$E\sim -nK_X-\lf(n+1)\mu(B)\rf-n\bM_X.$$
We claim that $(X,B+\bM)$ is $n$-complementary. 

\smallskip

We may write $B:=\sum_i b_iB_i$, and $\mu(B_i)=\sum_{j=1}^mm_{ij}D_j$, where $B_i$ are distinct prime divisors, $b_i\in[0,1]$, and $m_{ij}\in\Zz_{\ge0}$ for any $i,j$. Then 
\begin{align*}
\lf(n+1)\mu(B)\rf&=\sum_{j=1}^m\lf(n+1)\sum_im_{ij}b_i\rf D_j\\&
\ge\sum_{j=1}^m\sum_im_{ij}\lf(n+1)b_i\rf D_j=\mu\lf(n+1)B\rf.
\end{align*}
In particular, 
$$E':=\lf(n+1)\mu(B)\rf-\mu\lf(n+1)B\rf$$ 
is an effective Weil divisor. Since $B_\Phi=B_{\Phi_0}$ and $n\Phi_0\subseteq\Zz_{\ge0}$ by our choice, one has
$$\frac{\lf(n+1)B\rf}{n}\ge\frac{\lf(n+1)B_\Phi\rf}{n}\ge B_\Phi$$ 
and 
\begin{align*}
&\ \ \ -nK_X-\lf(n+1)B\rf-n\bM_X\sim -nK_X-\mu\lf(n+1)B\rf-n\bM_X\\&
=-nK_X-\lf(n+1)\mu(B)\rf-n\bM_X+\left(\lf(n+1)\mu(B)\rf-\mu\lf(n+1)B\rf\right).
\end{align*}
It follows that
$$-nK_X-\lf(n+1)B\rf-n\bM_X\sim E+E'.$$
Set 
$$B^+:=\frac{\lf(n+1)B\rf+E+E'}{n}\ge B_{\Phi},$$ 
then we have 
$$n\left(K_X+B^++\bM_X\right)\sim0.$$ 
Moreover, $(X,B^++\bM)$ is klt since $(X,B_\Phi+\bM)$ is exceptional. Therefore $(X,B^++\bM)$ is an $n$-complement of $(X,B+\bM)$. The proof is finished.
\end{proof}

\section{Boundedness of complements for semi-exceptional generalized pairs}\label{sec6}
In this section, we show the boundedness of complements for g-pairs which are semi-exceptional, cf. {\cite[Theorem 13]{Sho20}}.

\begin{thm}\label{thm:semiexc}
Let $d,p,I$ be positive integers, $\fR\subseteq[0,1]\cap\Qq$ a finite set, and $\Phi:=\Phi(\fR)$. Then there exists a finite set $\cN\subseteq\Zz_{>0}$ depending only on $d,p,I$, and $\Phi$ satisfying $I\mid \cN,$ and the following.

Assume that $(X,B+\bM)$ is a projective g-pair of dimension $d$ such that
\begin{enumerate}
  \item $X$ is of Fano type, 
  \item $p\bM$ is b-Cartier,
  \item $(X,B_{\Phi}+\bM)$ is semi-exceptional, and
  \item $(X,B+\bM)$ is $\Rr$-complementary.
\end{enumerate}
Then $(X,B+\bM)$ is $\cN$-complementary.
\end{thm}

We recall the definition of semi-exceptional, which is referred to as \emph{not strongly non-exceptional} in \cite[2.15]{Bir19}.

\begin{defn}\label{defn:semiexc}
We say that a projective g-pair $(X,B+\bM)$ is \emph{semi-exceptional}, if it is $\Rr$-complementary and $(X,B+G+\bM)$ is lc for any effective $\Rr$-divisor $G\sim_{\Rr}-(K_X+B+\bM_X)$.
\end{defn}

\begin{lem}\label{lem:pushdownissemiexc}
Let $(X,B+\bM)$ be a projective g-pair which is semi-exceptional. Suppose that $f:Y\to X$ is a birational morphism and $(Y,B_Y+\bM)$ is an $\Rr$-complementary g-pair such that $f_*B_Y\ge B$. Then $(Y,B_Y+\bM)$ is semi-exceptional. 
\end{lem}
\begin{proof}
It suffices to show that for any effective $\Rr$-divisor $G_Y\sim_\Rr-(K_Y+B_Y+\bM_Y)$, $(Y,B_Y+G_Y+\bM)$ is lc. To this end, let $G:=f_*G_Y$, then by assumption, $(X,B+G+\bM)$ is an $\Rr$-complement of $(X,B+\bM)$. Moreover, by the negativity lemma, we have
$$K_Y+B_Y+G_Y+\bM_Y=f^*(K_X+B+G+\bM_X),$$ 
and thus $(Y,B_Y^++\bM)$ is lc. This completes the proof.
\end{proof}

\begin{lem}\label{lem:semiexcnotbig}
Let $(X,B+\bM)$ be a projective g-pair. Suppose that $(X,B+\bM)$ is strictly lc (i.e. lc but not klt) and semi-exceptional. Then $-(K_X+B+\bM_X)$ is not big.
\end{lem}
\begin{proof}
Suppose on the contrary that $-(K_X+B+\bM_X)$ is big. Let $f:(Y,B_Y+\bM)\to (X,B+\bM)$ be a dlt modification of $(X,B+\bM)$. By Lemma \ref{lem:pushdownissemiexc}, $(Y,B_Y+\bM)$ is semi-exceptional. Possibly replacing $(X,B+\bM)$ by $(Y,B_Y+\bM)$, we may assume that $X$ is $\Qq$-factorial.

Let $(X,B+G+\bM)$ be an $\Rr$-complement of $(X,B+\bM)$ for some $\Rr$-divisor $G\ge0$, and $S$ a prime divisor that contains some non-klt center of $(X,B+\bM)$. Since $G\sim_\Rr-(K_X+B+\bM_X)$ is big, we may pick a small enough positive real number $\epsilon$ such that $G-\epsilon S\sim_{\Rr}F$ for some effective $\Rr$-divisor $F$. We derive a contradiction since then
$$K_X+B+\epsilon S+F+\bM_X\sim_\Rr K_X+B^++\bM_X\sim_\Rr0,$$
and $(X,B+\epsilon S+F+\bM)$ is not lc.
\end{proof}

The following lemma says that the semi-exceptionality is preserved under the canonical bundle formula.
\begin{lem}\label{lem:semiexccbf}
Suppose that $(X,B+\bM)$ is semi-exceptional, $B$ is a $\Qq$-divisor, and $\bM$ is b-$\Qq$-Cartier. Suppose that $\phi:X\to T$ is a contraction such that $K_X+B+\bM_X\sim_{\Qq,T}0$ and $\dim T>0$. Let $B_T$ and $\bM_\phi$ be the discriminant part and a moduli part of a canonical bundle formula for $(X,B+\bM)$ over $T$ respectively. Then $(T,B_T+\bM_\phi)$ is semi-exceptional.
\end{lem}
\begin{proof}
By Lemma \ref{lem:Qcompl}, $(X,B+\bM)$ has an $\Rr$-complement $(X,B+G+\bM)$ for some $\Qq$-Cartier $\Qq$-divisor $G\ge0$. Then $G\sim_{\Qq,T}0$ by assumption and thus $G=\phi^*G_T$ for some $\Qq$-Cartier $\Qq$-divisor $G_T\ge0$ on $T$ by \cite[Lemma 2.5]{CHL23}. Then we have
\begin{align*}
 K_X+B+G+\bM_X&\sim_\Qq\phi^*(K_T+B_T+G_T+\bM_{\phi,T})\sim_\Qq0
\end{align*}
which implies that $K_T+B_T+G_T+\bM_{\phi,T}\sim_\Qq0$ and $(T,B_T+G_T+\bM_\phi)$ is lc. In particular, $(T,B_T+G_T+\bM_\phi)$ is an $\Rr$-complement of $(T,B_T+\bM_\phi)$.

For any effective $\Rr$-Cartier $\Rr$-divisor $G_T'\sim_{\Rr}-(K_T+B_T+\bM_{\phi,T})$, let 
$$G':=\phi^*G_T'\sim_\Rr -(K_X+B+\bM_X).$$
By assumption $(X,B+G'+\bM)$ is lc and thus so is $(T,B_T+G_T'+\bM_\phi)$ which completes the proof.
\end{proof}

\begin{prop}\label{prop:CYindex}
Let $d,p,l$ be positive integers and $\Phi$ the hyperstandard set associated to a finite set $\fR\subseteq[0,1]\cap\Qq$. Then there exists a positive integer $n$ which is divisible by $l$ depending only on $d,p,l$ and $\Phi$ satisfying the following.

Assume that $(X,B+\bM)$ is a projective g-pair such that
\begin{enumerate}
  \item $\dim X=d$,
  \item $p\bM$ is b-Cartier,
  \item $X$ is of Fano type,
  \item $\kappa(B-B_0)=0$ for some $B_0\in\Phi,$ and 
  \item $(X,B+\bM)$ is an $\Rr$-complement of itself.
\end{enumerate}
Then $n(K_X+B+\bM_X)\sim0$.
\end{prop}

\begin{proof}
By \cite[Theorem 1.10]{Bir19} (see also \cite[Theorem 1.1]{Chen23}), there exists a positive integer $n$ which is divisible by $l$ depending only on $d,p,l$ and $\Phi$ such that if we further assume that $K_X+B_0+\bM_X\sim_\Qq0$, then $n(K_X+B_0+\bM_X)\sim0$. We claim that $n$ has the required properties.

We may run the $(B-B_0)$-MMP which terminates with a model $X'$ such that $B'=B'_0$ as $\kappa(B-B_0)=0$, where $B'$ and $B'_0$ are the strict transforms of $B$ and $B_0$ on $X'$ respectively. In particular we have
$$K_{X'}+B'_0+\bM_{X'}=K_{X'}+B'+\bM_{X'}\sim_\Qq0.$$
By our choice of $n$, we infer that
$$n(K_{X'}+B'+\bM_{X'})=n(K_{X'}+B'_0+\bM_{X'})\sim0.$$
Since $K_X+B+\bM_X\sim_\Qq0$, by the negativity lemma, one can see that
$$f^*(K_X+B+\bM_X)=f'^*(K_{X'}+B'+\bM_{X'})$$ for any common resolution $f:W\to X$ and $f':W\to X'$. Hence we conclude that $n(K_X+B+\bM_X)\sim0.$ This finish the proof.
\end{proof}

For convenience, by ``$\cN(\le d,p,I,\Phi)$ is given by Theorem \ref{thm:semiexc}'', we mean that the set $\cN(\le d,p,I,\Phi)$ satisfies the property of Theorem \ref{thm:semiexc}. Similarly, by ``$\Phi'(d,p,\Phi)$ is given by Proposition \ref{prop:cbfindex}'' we mean that the set $\Phi'(d,p,\Phi)$ satisfies the property of Proposition \ref{prop:cbfindex}.

\begin{proof}[Proof of Theorem \ref{thm:semiexc}]
We will perform a series of birational operations, and readers can refer to Figure \ref{fig1} to recall their connections.

We prove the theorem by induction on the dimension. Assume Theorem \ref{thm:semiexc} in dimension $\le d-1$. We will construct a finite sequence $p_i,$ $\Phi_i$ and $\Phi_i',$ $\cN_{i+1}$ and $\cN^{i+1}$ ($0\le i\le d-1$) of positive integers, hyperstandard sets, finite sets of positive integers respectively.

Set $\Phi_{0}:=\Phi$. Let $p_0:=p_0(d,p,\Phi_{0})$ be an integer and $\Phi'_{0}:=\Phi_0'(d,p,\Phi_{0})$ a hyperstandard set given by Proposition \ref{prop:cbfindex}. Let $\cN^1=\cN_1:=\cN_1(\le d-1,p_0,I,\Phi_{0}')\subseteq\Zz_{>0}$ be a finite set given by Theorem \ref{thm:semiexc}. Now suppose that $p_i,\Phi_{i},\Phi_{i}',\cN_{i+1}$ and $\cN^{i+1}$ are constructed for some $0\le i\le d-2$. Then we let 
$\Phi_{i+1}:=\Ii(\cN^{i+1},\Phi).$
Let $p_{i+1}:=p_{i+1}(d,p,\Phi_{i+1})$ be an integer and $\Phi'_{i+1}:=\Phi_{i+1}'(d,p,\Phi_{i+1})$ a hyperstandard set given by Proposition \ref{prop:cbfindex}. Let $\cN_{i+2}:=\cN_{i+2}(\le d-1,p_{i+1},I,\Phi_{i+1}')\subseteq\Zz_{>0}$ be a finite set given by Theorem \ref{thm:semiexc} and 
$$\cN^{i+2}:=\cN_{i+2}\cup\cN^{i+1}.$$ 
Now we finish the construction inductively. 

Set $\Phi_d:=\Gamma(\cN^d,\Phi)$ and let $\cN_d^{exc}:=\cN_d^{exc}(d,p,\Phi_d)\subseteq\Zz_{>0}$ be a finite set given by Theorem \ref{thm:exc}. Let $n_{CY}:=n_{CY}(d,p,\Phi_d)$ be a positive integer given by Proposition \ref{prop:CYindex}. We will show that $\cN^d\cup\cN_d^{exc}\cup\{n_{CY}\}$ has the required property.

By assumption and construction, $(X,B_{\Phi_d}+\bM)$ is semi-exceptional. If $(X,B_{\Phi_d}+\bM)$ is exceptional, then $(X,B+\bM)$ is $\cN_d^{exc}$-complementary by the choice of $\cN_d^{exc}$ and Lemma \ref{lem:N_Phicompl}. Hence we may assume that $(X,B_{\Phi_d}+\bM)$ is semi-exceptional and not exceptional. In particular, $(X,B_{\Phi_i}+\bM)$ is semi-exceptional and not exceptional for any $0\le i\le d$. In the following, we will prove that $(X,B+\bM)$ is $\cN^d\cup\{n_{CY}\}$-complementary.

Let $(X,B^++\bM)$ be an $\Rr$-complement of $(X,B_{\Phi_d}+\bM)$ which is not klt. By Lemma \ref{lem:Qcompl}, we may assume that $B^+\in\Qq$. Let $f:Y\to X$ be a dlt modification of $(X,B^++\bM)$, and we may write
$$K_Y+B_Y^++\bM_Y=f^*(K_X+B^++\bM_X)$$
for some $\Rr$-divisor $B_Y^+\ge0$. Note that $f_*((B_Y^+)_{\Phi_d})\ge B_{\Phi_d}$. Then by Lemma \ref{lem:pushdownissemiexc}, $(Y,(B_Y^+)_{\Phi_d}+\bM)$ is semi-exceptional. Possibly replacing $(X,B+\bM)$ by $(Y,B_Y^++\bM)$, we may assume that $(X,B+\bM)$ is a $\Qq$-factorial dlt CY g-pair. Note that $(X,B_{\Phi_i}+\bM)$ is dlt for any $i$.

According to Lemma \ref{lem:semiexcnotbig} we have 
$$0\le\kappa(B-B_{\Phi_d})\le\cdots\le\kappa(B-B_{\Phi_0})\le d-1.$$ 
It follows that there exists an integer $0\le i\le d-1$ such that 
$$\kappa(B-B_{\Phi_i})=\kappa(B-B_{\Phi_{i+1}})=:k.$$ 
If $k=0$, then $\kappa(B-B_{\Phi_d})=0$ and thus $(X,B+\bM)$ is $n_{CY}$-complementary by our choice of $n_{CY}$. Hence we may assume that $k\ge1$. We only need to show that $(X,B_{\Phi_{i+1}}+\bM)$ is $\cN_{i+1}$-complementary and thus so is $(X,B+\bM)$.

We may run an MMP on $B-B_{\Phi_{i+1}}$ which terminates with a model $\psi:X\dashto X'$ on which $B'-B_{\Phi_{i+1}}'$ is semi-ample, where $D'$ denotes the strict transform of $D$ on $X'$ for any $\Rr$-divisor $D$ on $X$. Let  $X'\to Z'$ be the induced morphism by $B'-B_{\Phi_{i+1}}'$. Since this MMP is also an MMP on $-(K_X+B_{\Phi_{i+1}}+\bM_X)$,
$$a\left(P,X',B_{\Phi_{i+1}}'+\bM\right)<a(P,X,B_{\Phi_{i+1}}+\bM)\le1$$ 
for any $\psi$-exceptional prime divisor $P$. Then by \cite[Theorem 1.7]{LX23}, there exists a birational morphism $f':Y'\to X'$ extracts all the $\psi$-exceptional prime divisors on $X$, i.e., the birational map $\varphi:Y'\dashto X$ is isomorphism in codimenison one. We may write
$$K_{Y'}+B_{\Phi_{i+1},Y'}+\bM_{Y'}=f'^*\left(K_{X'}+B_{\Phi_{i+1}}'+\bM_{X'}\right)$$
and
$$K_{Y'}+B_{Y'}+\bM_{Y'}=f'^*\left(K_{X'}+B'+\bM_{X'}\right)$$
for some effective $\Qq$-divisors $B_{Y'}$ and $B_{\Phi_{i+1},Y'}$. It is clear that $B_{Y'}\ge B_{\Phi_{i+1},Y'},$
$$\varphi_*B_{Y'}=B,\text{ and } \varphi_* B_{\Phi_{i+1},Y'}\ge B_{\Phi_{i+1}}.$$ 
By Lemma \ref{lem:pushdownissemiexc}, $(Y',(B_{\Phi_{i+1},Y'})_{\Phi_{i}}+\bM)$ is semi-exceptional as $\varphi_* (B_{\Phi_{i+1},Y'})_{\Phi_i}\ge (B_{\Phi_{i+1}})_{\Phi_i}=B_{\Phi_i}$ and $(X,B_{\Phi_i}+\bM)$ is also semi-exceptional. Now 
$$\kappa\left(B_{Y'}-(B_{\Phi_{i+1},Y'})_{\Phi_{i}}\right)\ge\kappa\left(B_{Y'}-B_{\Phi_{i+1},Y'}\right)=\kappa\left(B'-B'_{\Phi_{i+1}}\right).$$ 
On the other hand, by Lemma \ref{lem:smiitdim}, we see that
$$\kappa\left(B_{Y'}-(B_{\Phi_{i+1},Y'})_{\Phi_{i}}\right)=\kappa\left(B-\varphi_* (B_{\Phi_{i+1},Y'})_{\Phi_i}\right)\le \kappa(B-B_{\Phi_{i}})$$ 
which implies that 
$$\kappa\left(B_{Y'}-(B_{\Phi_{i+1},Y'})_{\Phi_{i}}\right)=\kappa\left(B_{Y'}-B_{\Phi_{i+1},Y'}\right).$$
\begin{figure}[htbp]
    \centering
	\begin{tikzcd}[column sep = 2em, row sep = 2em]
	{}&& Y'\arrow[dd, "\phi'", bend left]\arrow[d,"f'" swap]\arrow[rr,dashed] &&Y''\arrow[dd,"\phi''"]\\
		X \arrow[rr,"\psi", dashed]  && X' \arrow[d, "\pi'" swap] &&{} \\
		{} && Z'  && \arrow[ll, "\tau''" swap]Z'' 
	\end{tikzcd}
    \caption{} 
    \label{fig1} 
\end{figure}

Note that $Y'$ is also of Fano type. Then we may run a $(B_{Y'}-(B_{\Phi_{i+1},Y'})_{\Phi_i})$-MMP$/Z'$ and it terminates with a model $Y''$ on which $B_{Y''}-(B_{\Phi_{i+1},Y'})_{\Phi_i,Y''}$ is semi-ample$/Z'$, where $D_{Y''}$ is the strict transform of $D$ on $Y''$ for any $\Rr$-divisor $D$ on $Y'$. Suppose that $\phi'':Y''\to Z''$ is the induced morphism by $B_{Y''}-(B_{\Phi_{i+1},Y'})_{\Phi_i,Y''}$ over $Z'$. Then 
$$K_{Y''}+(B_{\Phi_{i+1},Y'})_{\Phi_i,Y''}+\bM_{Y''}\sim_{\Qq,Z''}0.$$ 
Moreover, $\tau'':Z''\to Z'$ is a birational morphism by Lemma \ref{lem:twommp}. 

By Proposition \ref{prop:cbfindex} and our construction of $\Phi_i'$ and $p_i$, there exists a g-pair $(Z'',B_{Z''}^{(i)}+\bM_{\phi''})$ such that
\begin{itemize}
  \item $B_{Z''}^{(i)}\in\Phi_i'$, 
  \item $p_i\bM_{\phi''}$ is b-Cartier, and
  \item $p_i\left(K_{Y''}+(B_{\Phi_{i+1},Y'})_{\Phi_i,Y''}+\bM_{Y''}\right)\sim p_i\phi''^*\left(K_{Z''}+B_{Z''}^{(i)}+\bM_{\phi'',Z''}\right)$.
\end{itemize}
Since $K_{X'}+B_{\Phi_{i+1}}'+\bM_{X'}\sim_{\Qq,Z'}0$, $K_{Y'}+B_{\Phi_{i+1},Y'}+\bM_{Y'}\sim_{\Qq,Z'}0$ and hence 
$$K_{Y''}+B_{\Phi_{i+1},Y''}+\bM_{Y''}\sim_{\Qq,Z'}0.$$ 
In particular, $0\le B_{\Phi_{i+1},Y''}-(B_{\Phi_{i+1},Y'})_{\Phi_i,Y''}\sim_{\Qq,Z''}0$ which implies that $B_{\Phi_{i+1},Y''}-(B_{\Phi_{i+1},Y'})_{\Phi_i,Y''}=\phi''^*G_{Z''}$ for some $\Qq$-Cartier $\Qq$-divisor on $Z''$ (see \cite[Lemma 2.5]{CHL23}). Thus we can choose $\bM_{\phi''}$ to be a moduli part of a canonical bundle for $(Y'',B_{\Phi_{i+1},Y''}+\bM)$ over $Z''$ (cf. \cite[Lemma 3.5]{Bir19}). By the canonical bundle formula, there exists a g-pair $(Z'',B_{Z''}^{(i+1)}+\bM_{\phi''})$ such that
$$K_{Y''}+B_{\Phi_{i+1},Y''}+\bM_{Y''}\sim_{\Qq}\phi''^*\left(K_{Z''}+B_{Z''}^{(i+1)}+\bM_{\phi'',Z''}\right).$$ 
Moreover, 
\begin{align*}
&\ \ \ \ p_i\left(K_{Y''}+B_{\Phi_{i+1},Y''}+\bM_{Y''}\right)
\\&=p_i\left(K_{Y''}+(B_{\Phi_{i+1},Y'})_{\Phi_i,Y''}+\bM_{Y''}\right)+p_i\left(B_{\Phi_{i+1},Y''}-(B_{\Phi_{i+1},Y'})_{\Phi_i,Y''}\right)
\\&\sim p_i\phi''^*\left(K_{Z''}+B_{Z''}^{(i)}+\bM_{\phi'',Z''}\right)+p_i\phi''^*(G_{Z''})
\\&\sim p_i\phi''^*\left(K_{Z''}+B_{Z''}^{(i+1)}+\bM_{\phi'',Z''}\right).
\end{align*}
Since $(Y',B_{\Phi_{i+1},Y'}+\bM)$ is crepant to $(Y'',B_{\Phi_{i+1},Y''}+\bM)$, we can choose $\bM_{\phi''}=\bM_{\phi'}$ to be a moduli part of a canonical bundle for $(Y',B_{\Phi_{i+1},Y'}+\bM)$ over $Z'$. Then by the canonical bundle formula, there exists a g-pair $(Z',B_{Z'}^{(i+1)}+\bM_{\phi'})$ such that
$$K_{Y'}+B_{\Phi_{i+1},Y'}+\bM_{Y'}\sim_\Qq \phi'^*\left(K_{Z'}+B_{Z'}^{(i+1)}+\bM_{\phi',Z'}\right).$$ 
Moreover, we have that
$$K_{Z''}+B_{Z''}^{(i+1)}+\bM_{\phi'',Z''}=\tau''^*\left(K_{Z'}+B_{Z'}^{(i+1)}+\bM_{\phi',Z'}\right)$$
and
$$p_i(K_{Y'}+B_{\Phi_{i+1},Y'}+\bM_{Y'})\sim p_i\phi'^*\left(K_{Z'}+B_{Z'}^{(i+1)}+\bM_{\phi',Z'}\right).$$ 

By Lemma \ref{lem:semiexccbf}, $(Z'',B^{(i+1)}_{Z''}+\bM_{\phi'})$ is semi-exceptional. By Lemma \ref{lem:cbfcptoverbase}, one can find a crepant model $(Z''',B^{(i+1)}_{Z'''}+\bM_{\phi'})$ of $(Z'',B^{(i+1)}_{Z''}+\bM_{\phi'})$ with induced morphism $\tau''':Z'''\to Z''$ such that for any prime divisor $P\subseteq\Supp B_{\Phi_{i+1},Y'}$ which is vertical over $Z'$, $\Center_{Z'''}P$ is a prime divisor. Since $B_{Z''}^{(i)}\in\Phi_i'$,
$$\tau'''_*(\left(B^{(i+1)}_{Z'''}\right)_{\Phi_i'}= \left(B^{(i+1)}_{Z''}\right)_{\Phi_i'}\ge\left(B_{Z''}^{(i)}\right)_{\Phi_i'}=B_{Z''}^{(i)},$$ 
and therefore $(Z''',(B^{(i+1)}_{Z'''})_{\Phi_i'}+\bM_{\phi'})$ is semi-exceptional by Lemma \ref{lem:pushdownissemiexc}. Therefore by induction, $(Z''',B^{(i+1)}_{Z'''}+\bM_{\phi'})$ is $n$-complementary for some $n\in\cN_{i+1}$. It follows by Proposition \ref{prop:cbflift1} that $({Y'},B_{\Phi_{i+1},Y'}+\bM)$ is also $n$-complementary. We conclude that $(X,B_{\Phi_{i+1}}+\bM)$ is $n$-complementary as $X$ and $Y'$ are isomorphic in codimension one. 
\end{proof}

\section{Boundedness of complements for generalized pairs of generic klt type}\label{sec7}
In this section, we show the boundedness of complements for g-pairs of generic klt type, cf. {\cite[Theorem 14]{Sho20}}.

\begin{thm}\label{thm:generictype}
Let $d,p,I$ be positive integers, $\fR\subseteq[0,1]\cap\Qq$ a finite set, and $\Phi:=\Phi(\fR)$. Then there exists a finite set $\mathcal{N}$ of positive integers depending only on $d,p, I,$ and $\Phi$ satisfying $I\mid\cN$, and the following.

Assume that $(X/Z\ni z,B+\bM)$ is a g-pair of dimension $d$ such that $p\bM$ is b-Cartier, and $(X/Z\ni z,B_{\cN\_\Phi}+\bM)$ is of generic klt type. Then $(X/Z\ni z,B+\bM)$ is $\mathcal{N}$-complementary.
\end{thm}

\begin{defn}\label{defn:gklt}
We say that a g-pair $(X/Z\ni z,B+\bM)$ is of $\emph{generic klt type}$, if it has an $\Rr$-complement $(X/Z\ni z,B^++\bM)$ such that over a neighborhood of $z$, $B^+-B$ is big and $(X,B^++\bM)$ is klt.
\end{defn}

\begin{rem}
Assume that $(X/Z\ni z,B+\bM)$ is of generic klt type, then $X$ is of Fano type over a neighborhood of $z$. 
\end{rem}

\begin{lem}\label{lem:pltcomp1}
Assume that $(X/Z\ni z,B+\bM)$ is of generic klt type and not semi-exceptional. Then $(X/Z\ni z,B+\bM)$ has an $\Rr$-complement $(X/Z\ni z,B^++\bM)$ with only one lc place $S$ such that $z\in\Center_Z{S}$, and $\Supp(B^+-B)$ supports an effective divisor $A$ which is ample over a neighborhood of $z$.
\end{lem}
\begin{proof}
By assumption and possibly shrinking $Z$ around $z$, we may find a klt g-pair $(X/Z,B'+\bM)$ such that $K_X+B'+\bM_X\sim_{\Rr,Z}0$ and  $B'-B\ge0$ is big over $Z$. We may assume that that $B'-B\sim_{\Rr,Z}A+E$ for some $\Rr$-divisor $E\ge0$ and ample over $Z$ $\Rr$-divisor $A\ge0$. Since $(X/Z\ni z,B+\bM)$ is not semi-exceptional, possibly shrinking $Z$ near $z$, there exists an $\Rr$-divisor $C\ge B$ such that $(X/Z,C+\bM)$ is not lc and $K_X+C+\bM\sim_{\Rr,Z}0$.

Let $f:W\to X$ be a log resolution of $(X,\Supp(B+C+E))$ such that $\bM$ descents to $W$. We may write
$$K_W+B_W+\bM_W:=f^*(K_X+B+\bM_X),$$
$$K_W+C_W+\bM_W:=f^*(K_X+C+\bM_X),$$
and
$$K_W+B_W'+\bM_W:=f^*(K_X+B'+\bM_X)$$
for some $\Rr$-divisors $C_W,B_W$, and $B_W'$. Let $H$ be an ample over $Z$ $\Rr$-divisor on $W$ such that $-H\ge0$, $H$ is exceptional over $X$, and $A_W:=H+f^*{A}$ is ample over $Z$.

Let $\epsilon$ be a positive real number such that $||C_{\epsilon}^{>0}||_{\infty}>1$, where
$$C_{\epsilon}:=(1-\epsilon)C_W+\epsilon B_W+\epsilon f^*E-\epsilon H\ge B_W.$$
One can find a positive real number $t$ such that 
$$\norm{(tC_\epsilon+(1-t)B_W')^{>0}}_\infty=1\text{ and }\lf(tC_\epsilon+(1-t)B_W')^{>0}\rf=S$$
is a prime divisor. In particular, $(W,tC_\epsilon+(1-t)B_W'+\bM)$ is sub-lc with only one lc place $S$. We may assume that $\Supp B_W^+$ is snc, and $\lf(B_W^+)^{>0}\rf=S$, where
$$B_W^+:=tC_\epsilon+(1-t)B_W'+t\epsilon A_W.$$
Notice that
\begin{align*}
C_{\epsilon}+\epsilon A_W&=(1-\epsilon)C_W+\epsilon B_W+\epsilon(f^*E-H)+\epsilon(f^*A+H)
\\&\sim_{\Rr,Z} (1-\epsilon)C_W+\epsilon B_W+\epsilon(B_W'-B_W)=(1-\epsilon)C_W+\epsilon B'_W
\\&\sim_{\Rr,Z}C_W.
\end{align*}
Therefore 
$$K_W+B_W^++\bM_W\sim_{\Rr,Z}K_W+(1-t)B'_W+tC_W+\bM_W\sim_{\Rr,Z}0.$$ 
Let $B^+:=f_*B_W^+$. Then we have $K_X+B^++\bM_X\sim_{\Rr,Z}0$ and
$$K_W+B_W^++\bM_W=f^*(K_X+B^++\bM_X)$$ 
by the negativity lemma. Hence $(X/Z\ni z,B^++\bM)$ has the required properties.
\end{proof}

\begin{lem}\label{lem:pltcomp2}
Assume that $(X/Z\ni z,B+\bM)$ is of generic klt type and not semi-exceptional. Then there exists a birational morphism $f:Y\to X$, and a g-pair $(Y/Z\ni z,B_Y+\bM)$ such that over a neighborhood of $z$, we have
\begin{enumerate}
    \item $(Y,B_Y+\bM)$ is plt,
    \item $-(K_Y+B_Y+\bM_Y)$ is ample,
    \item $K_Y+B_Y+\bM_Y\ge f^*(K_X+B+\bM_X)$, and
    \item $\lf B_Y\rf=S$ is the only possible $f$-exceptional prime divisor.
\end{enumerate}
\end{lem}
\begin{proof}
By Lemma \ref{lem:pltcomp1} and shrinking $Z$ around $z$, $(X/Z\ni z,B+\bM)$ has an $\Rr$-complement $(X/Z\ni z,B'+\bM)$ with only one lc place $S$ such that $B'-B\ge A\ge0$ for some ample over $Z$ $\Rr$-divisor $A$. Suppose that $(X/Z\ni z,B^++\bM)$ is an $\Rr$-complement of $(X/Z\ni z,B+\bM)$ such that $(X/Z,B^++\bM)$ is klt.

Let $f:Y\to X$ be a birational morphism that only extracts $S$. Note that if $S$ is a prime divisor on $X$, then $f=id$. Let $B'':=B'-A\ge B\ge0$. We may write 
$$K_Y+B_Y''+\bM_Y:=f^*\left(K_X+B''+\bM_X\right),$$
and
$$K_Y+B_Y^++\bM_Y:=f^*\left(K_X+B^++\bM_X\right)$$
for some $\Rr$-divisors $B''_Y$ and $B_Y^+$.

Let $b_S^+:=\mult_SB_Y^+$. For any $0\le\epsilon\ll1$.
Set
$$t:=\frac{\epsilon}{\epsilon+1-b_S^+} \text{ and }B_Y^\epsilon:=(1-t)B_Y''+tB_Y^+.$$
We claim that one can pick a positive real number $\epsilon$ such that $(Y/Z,B^\epsilon_Y+\bM)$ is plt and
$$-(K_Y+B_Y^\epsilon+\bM_Y)\sim_{\Rr,Z}(1-t)(f^*A-\epsilon S)$$
is ample over $Z$. It suffices to prove the claim as the statement follows immediately.

\smallskip

It is obvious that $f^*A-\epsilon S$ is ample over $Z$ for any $0<\epsilon\ll1$. For any prime divisor $S'\neq S$ on $Y$, we have
$$\mult_{S'}B_Y^\epsilon=(1-t)\mult_{S'}B_Y''+t\mult_{S'}B_Y^+<1,$$
as $(X/Z,B^++\bM)$ is klt and $(X/Z,B''+\bM)$ is lc with at most one lc place $S$. Let $g:W\to Y$ be a log resolution of $(Y,B_Y^\epsilon)$. We only need to show that for any $g$-exceptional prime divisor $P$ on $W$, we have that
$$(1-t)\left(b_P''+\epsilon s_P\right)+tb_P^+<1,$$
where $s_P:=\mult_Pg^*S$, 
$$b_P'':=\mult_P\left(g^*\left(K_Y+B_Y''+\bM_Y\right)-K_W-\bM_W\right),$$
and
$$b_P^+:=\mult_P\left(g^*\left(K_Y+B_Y^++\bM_Y\right)-K_W-\bM_W\right).$$ 
Since $t=\frac{\epsilon}{\epsilon+1-b_S^+},$ we only need to show
$$h(\epsilon):=\left(1-\frac{\epsilon}{\epsilon+1-b_S^+}\right)\left(b_P''+\epsilon s_P\right)+\frac{\epsilon}{\epsilon+1-b_S^+}b_P^+<1$$
for $0<\epsilon\ll1$. Note that $(Y/Z\ni z,B_Y''+\bM)$ is plt and hence $h(0)=b_P''<1$. Since $h(\epsilon)$ is a continuous function near $0$, $h(\epsilon)<1$ for $0<\epsilon\ll1$. This completes the proof.
\end{proof}

\begin{proof}[Proof of Theorem \ref{thm:generictype}]
We will perform a series of birational operations, and readers can refer to Figure \ref{fig2} to recall their connections.

We prove the theorem by induction on the dimension. Assume Theorem \ref{thm:generictype} holds in dimension $=d-1$. Let $\Phi':=\Phi'(p,\Phi)$ be a hyperstandard set given by Proposition \ref{prop:adjcoeff}, and $\cN^{d-1}:=\cN^{d-1}(d-1,p,I,\Phi')\subseteq\Zz_{>0}$ a finite set given by Theorem \ref{thm:generictype}. Let $\cN_d:=\cN_d(d,p,I,\Ii(\cN^{d-1},\Phi))\subseteq\Zz_{>0}$ be a finite set given by Theorem \ref{thm:semiexc}. We will show that $\mathcal{N}^d:=\mathcal{N}_d\cup\mathcal{N}^{d-1}$ has the required properties.

Possibly replacing $X$ with a $\Qq$-factorization we may assume that $X$ is $\Qq$-factorial. If $(X/Z\ni z,B_{\mathcal{N}^{d-1}{\_}\Phi}+\bM)$ is semi-exceptional, then $\dim Z=\dim z=0$ and $(X,B+\bM)$ is $\mathcal{N}_d$-complementary by our choice of $\mathcal{N}_d$. Hence we may assume that $(X/Z\ni z,B_{\mathcal{N}^{d-1}{\_}\Phi}+\bM)$ is not semi-exceptional. We claim that $(X/Z\ni z,B_{\mathcal{N}^{d-1}{\_}\Phi}+\bM)$ is $\mathcal{N}^{d-1}$-complementary and thus $(X/Z\ni z,B+\bM)$ is also $\mathcal{N}^{d-1}$-complementary by Lemma \ref{lem:N_Phicompl}. 

Possibly replacing $B$ with $B_{\mathcal{N}^{d-1}{\_}\Phi}$, we may assume that $B\in\Ii(\cN^{d-1},\Phi)$, $(X/Z\ni z,B+\bM)$ is of generic klt type and not semi-exceptional. By Lemma \ref{lem:pltcomp2} and possibly shrinking $Z$ near $z$, we may assume that 
\begin{itemize}
  \item $-(K_X+B+\bM_X)$ is ample over $Z$,
  \item $(X/Z,B+\bM)$ is plt, and
  \item $\lf B\rf=S$ is a prime divisor. 
\end{itemize}
Note that $S\to \pi(S)$ is a contraction and is of Fano type (cf. Step 1 of the proof of \cite[Proposition 8.1]{Bir19}), where $\pi:X\to Z$. Let $(S/\pi(S),B_S+\bM^\iota)$ be the divisorial adjunction of $(X/Z,B+\bM)$ on $S$. It is easy to see that $(S/\pi(S)\ni z,B_S+\bM^\iota)$ is of generic klt type. Let $(\tilde{S}/\pi(S),B_{\tilde{S}}+\bM^\iota)$ be a terminal model of $(S/\pi(S),B_S+\bM^\iota)$ (cf. \cite[Lemma 2.10(1)]{CT23}). Also note that $(\tilde{S}/\pi(S),B_{\tilde{S}}+\bM^\iota)$ is of generic klt type. By induction, $(\tilde{S},B_{\tilde{S}}+\bM^\iota)$ has an $n$-complement over a neighborhood of $z$ for some $n\in\cN^{d-1}$. In particular, $(\tilde{S},(B_{\tilde{S}})_{n\_\Phi'}+\bM^\iota)$ has a monotonic $n$-complement over a neighborhood of $z$ by Lemma \ref{lem:N_Phicompl}. We will show that $(X/Z\ni z,B+\bM)$ is $n$-complementary.

We may run a $-(K_X+B_{n\_\Phi}+\bM_X)$-MMP over $Z$ and it terminates with a model $\pi':X'\to Z$ such that $-(K_{X'}+B'_{n\_\Phi}+\bM_{X'})$ is big and nef over $Z$ as $-(K_X+B+\bM_X)$ is ample$/Z$ and $B\ge B_{n\_\Phi}$, where $B'_{n\_\Phi}$ is the strict transform of $B_{n\_\Phi}$ on $X'$. Since $a(S,X,B_{n\_\Phi}+\bM)=0$ and $(X',B'_{n\_\Phi}+\bM)$ is lc, $S$ is not contracted in this MMP. Furthermore, by Lemma \ref{lem:S'isFT}, $S'$ the strict transform of $S$ on $X'$ is of Fano type over $\pi'(S')$.

Let $f:W\to X$ and $f':W\to X'$ be a common resolution. We now claim that 
\begin{equation}\label{1}
    f^*(K_X+B+\bM_X)\ge f'^*\left(K_{X'}+B'_{n\_\Phi}+\bM_{X'}\right).
\end{equation}
Indeed, as 
$$-\left(f^*(K_X+B+\bM_X)-f'^*\left(K_{X'}+B'_{n\_\Phi}+\bM_{X'}\right)\right)$$ 
is nef over $X'$, and
$$f'_*f^*(K_X+B+\bM_X)-\left(K_{X'}+B'_{n\_\Phi}+\bM_{X'}\right)\ge0,$$ 
the inequality follows from the negativity lemma. 
\begin{figure}
\centering
\begin{tikzcd}[column sep = 4em, row sep = 1em]
	& W \arrow[dl, "f" swap] \arrow[dr, "f'"]\\
	X \arrow[rr, dashed] && X' \\
	& S_W\arrow[uu,hook] \arrow[dl,  "f|_{S_W}" swap] \arrow[dr,"f'|_{S_W}"]\\
	S \arrow[uu,hook] \arrow[rr, dashed, ""] && S' \arrow[uu,hook] \\
		& \tilde{S}\arrow[ul,  swap] \arrow[ur,dashed, "\phi" swap]
\end{tikzcd}
\caption{} 
    \label{fig2} 
\end{figure}
Now, let $S_W$ be the strict transform of $S$ on $W$. By \eqref{1}, we have
$$f|_{S_W}^*\left(K_S+B_S+\bM_S^\iota\right)\ge f'|_{S_W}^*\left(K_{S'}+(B'_{n\_\Phi})_{S'}+\bM_{S'}^\iota\right),$$
where $({S'},(B'_{n\_\Phi})_{S'}+\bM^\iota)$ is the divisorial adjunction of $({X'},B'_{n\_\Phi}+\bM)$ on $S'$. Thus for any prime divisor $D'$ over $S'$ with $a(D',S',(B'_{n\_\Phi})_{S'}+\bM^\iota)\le1$, it holds that
$$a\left(D',S,B_S+\bM^\iota\right)\le a\left(D',S',(B'_{n\_\Phi})_{S'}+\bM^\iota\right)\le1.$$ 
Hence the birational map $\phi:\tilde{S}\dashto S'$ is a birational contraction. It is clear that $\phi_*B_{\tilde{S}}\ge (B'_{n\_\Phi})_{S'}$. Moreover, as $(B'_{n\_\Phi})_{S'}\in\Ii(n,\Phi')$ by Proposition \ref{prop:adjcoeff}, we know that $\phi_*(B_{\tilde{S}})_{n\_\Phi'}\ge(B'_{n\_\Phi})_{S'}$. Therefore, over a neighborhood of $z$,  $(S',(B'_{n\_\Phi})_{S'}+\bM^\iota)$ has a monotonic $n$-complement since $(\tilde{S},(B_{\tilde{S}})_{n\_\Phi'}+\bM^\iota)$ has a monotonic $n$-complement. According to Proposition \ref{prop:adjlift}, $(X'/Z\ni z,B'_{n\_\Phi}+\bM)$ has a monotonic $n$-complement. It immediately implies that $(X/Z\ni z,B_{n\_\Phi}+\bM)$ has a monotonic $n$-complement. The proof is finished.
\end{proof}

\begin{rem}
In Theorem \ref{thm:generictype}, we do not assume that $(X/Z\ni z,B+\bM)$ is $\Rr$-complementary. Indeed, if $(X/Z\ni z,B_{\cN\_\Phi}+\bM)$ is $\cN$-complementary then so is $(X/Z\ni z,B+\bM)$ by Lemma \ref{lem:N_Phicompl}.
\end{rem}

\section{Proof of Theorem \ref{thm:main}}\label{sec8}

In this section, we prove our main result. Indeed, Theorem \ref{thm:main} follows from the following theorem immediately.

\begin{thm}\label{thm:klttype}
Let $d,p,I$ be positive integers, $\fR\subseteq[0,1]\cap\Qq$ a finite set, and $\Phi:=\Phi(\fR)$. Then there exists a finite set $\cN\subseteq\Zz_{>0}$ depending only on $d,p,I,$ and $\Phi$ satisfying $I\mid\cN,$ and  the following.

Assume that $(X/Z\ni z,B+\bM)$ is a g-pair such that
\begin{enumerate}
 \item $\dim X=d$,
 \item $p\bM$ is b-Cartier,
 \item $X$ is of Fano type over $Z$, and
 \item $(X/Z\ni z,B_{\cN\_\Phi}+\bM)$ has an $\Rr$-complement which is klt over a neighborhood of $z$.
\end{enumerate}
Then $(X/Z\ni z,B+\bM)$ is $\cN$-complementary.
\end{thm}

\begin{proof}
We prove the theorem by induction on the dimension. Assume Theorem \ref{thm:klttype} in dimensions $\le d-1$. We will construct a finite sequence $p_i,\Phi_i\text{ and }\Phi_i',\cN_{i+1}\text{ and }\cN^{i+1}$ ($i=0,\dots,d-1$) of positive integers, hyperstandard sets, finite sets of positive integers respectively.

Let $\cN^0=\cN_0:=\cN_0(d,p,I,\Phi)\subseteq\Zz_{>0}$ be a finite set given by Theorem \ref{thm:generictype}, and $\Phi_0:=\Ii(\cN^0,\Phi)$. Let $p_0:=p_0(d,p,\Phi_0)$ be a positive integer and $\Phi_0':=\Phi_0'(d,p_0,\Phi_0)$ a hyperstandard set given by Proposition \ref{prop:cbfindex}. Let $\cN_1:=\cN(\le d-1,p_0,I,\Phi_0')\subseteq\Zz_{>0}$ be a finite set given by Theorem \ref{thm:klttype}, and $\cN^1:=\cN_1\cup\cN^0$. Now suppose that $p_i,\Phi_i$ and $\Phi_i',\ \cN_{i+1}$ and $\cN^{i+1}$ are constructed for some $0\le i\le d-2$. Let $\Phi_{i+1}:=\Gamma(\cN^{i+1},\Phi)$. Let $p_{i+1}:=p_{i+1}(d,p,\Phi_{i+1})$ be a positive integer and $\Phi_{i+1}':=\Phi_{i+1}'(d,p,\Phi_{i+1})$ a hyperstandard set given by Proposition \ref{prop:cbfindex}. Let $\cN_{i+2}:=\cN(\le d-1,p_{i+1},\Phi_{i+1}')\subseteq\Zz_{>0}$ be a finite set given by Theorem \ref{thm:klttype}, and set 
$$\cN^{i+2}:=\cN_{i+2}\cup\cN^{i+1}.$$ 
Now we finish the construction inductively. 

Set $\Phi_d:=\Gamma(\cN^d,\Phi)$ and let $n_{CY}:=n_{CY}(d,p,\Phi_d)$ be a positive integer given by Proposition \ref{prop:CYindex}. We will show that $\cN:=\cN^d\cup\{n_{CY}\}$ has the required properties.

Possibly replacing $X$ with a small $\Qq$-factorialization, we may assume that $X$ is $\Qq$-factorial. By Lemma \ref{lem:Qcompl}, $(X/Z\ni z,B_{\cN\_\Phi}+\bM)$ has an $\Rr$-complement $(X/Z\ni z,B^++\bM)$ that is klt over a neighborhood of $z$ and $B^+\in\Qq$. Notice that $(X/Z\ni z,B+\bM)$ is $\cN$-complementary provided that $(X/Z\ni z,B^++\bM)$ is $\cN$-complementary. Therefore possibly shrinking $Z$ near $z$ and replacing $B$ by $B^+$, we may assume that $(X/Z,B+\bM)$ is klt and $K_X+B+\bM_X\sim_{\Qq,Z}0$.

If $(X/Z\ni z,B_{\Phi_0}+\bM)$ is of generic klt type, that is $B-B_{\Phi_0}$ is big over $Z$, then it is $\cN_0$-complementary by our choice of $\cN_0$ and we are done. Therefore we may assume that 
$$0\le\kappa(B-B_{\Phi_d}/Z)\le\cdots\le\kappa(B-B_{\Phi_0}/Z)\le d-1-\dim Z.$$
In particular, there exists an integer $0\le i\le d-1$ such that 
$$\kappa(B-B_{\Phi_i}/Z)=\kappa(B- B_{\Phi_{i+1}}/Z)=:k\ge0.$$ 
If $\dim Z=0$ and $k=0$, then $\kappa(B-B_{\Phi_d})=0$ which follows that 
$$n_{CY}(K_X+B+\bM_X)\sim0$$ 
by Proposition \ref{prop:CYindex}, and the statement follows. Thus we may assume that either $\dim Z>0$ or $k>0$. In the following, we will show that $(X/Z\ni z,B_{\Phi_{i+1}}+\bM)$ is $\cN_{i+1}$-complementary and thus finish the proof by Lemma \ref{lem:N_Phicompl}.

The rest of the proof is very similar as the proof of Theorem \ref{thm:semiexc}.
By Lemma \ref{lem:twommp}, we have the following commutative diagram
\begin{center}
	\begin{tikzcd}[column sep = 2em, row sep = 2em]
		X \arrow[d,"\pi",swap]\arrow[rr, "", dashed]  && X' \arrow[d, "\pi'" swap] \arrow[rr, dashed] && X'' \arrow[d, "{\pi''}" swap] \\
		Z && \arrow[ll, ""] Z'  && \arrow[ll, "\tau''", swap]Z'' 
	\end{tikzcd}
\end{center}
such that
\begin{itemize}
  \item $X\dashto X'$ is the $(B-B_{\Phi_{i+1}})$-MMP over $Z$, and $B'-B_{\Phi_{i+1}}'$ is semi-ample$/Z$,
  \item $X'\dashto X''$ is the $(B_{\Phi_{i+1}}'-B_{\Phi_i}')$-MMP over $Z'$ and $B_{\Phi_{i+1}}''-B_{\Phi_i}''$ is semi-ample$/Z'$,
  \item $\pi:X'\to Z'$ and $\pi'':X''\to Z''$ are the morphisms induced by $B'-B_{\Phi_{i+1}}'$ over $Z$ and $B_{\Phi_{i+1}}''-B_{\Phi_i}''$ over $Z'$ respectively, and 
  \item $\tau'':Z''\to Z'$ is a birational contraction between varieties of positive dimensions.
\end{itemize}
where $D'$ (resp. $D''$) denotes its strict transform on $X'$ (resp. $X''$) for any $\Rr$-divisor $D$ on $X$. By Proposition \ref{prop:cbfindex}, there exist g-pairs $(Z'/Z,B_{Z'}^{(i+1)}+\bM_{\pi'})$, $(Z''/Z,B_{Z''}^{(i)}+\bM_{\pi''})$ and $(Z''/Z,B_{Z''}^{(i+1)}+\bM_{\pi''})$ such that
\begin{itemize}
  \item $B_{Z''}^{(i)}\in\Phi_i'$, $p_{i}\bM_{\pi''}$ is b-Cartier,
  \item $K_{X'}+B_{\Phi_{i+1}}'+\bM_{X'} \sim_{\Qq}(\pi')^*\left(K_{Z'}+B_{Z'}^{(i+1)}+\bM_{\pi',Z'}\right)$,
  \item $p_{i}\left(K_{X''}+B_{\Phi_{i}}''+\bM_{X''}\right)\sim p_{i}(\pi'')^*\left(K_{Z''}+B_{Z''}^{(i)}+\bM_{\pi'',Z''}\right)$,
  \item $p_{i}\left(K_{X''}+B_{\Phi_{i+1}}''+\bM_{X''}\right)\sim p_{i}(\pi'')^*\left(K_{Z''}+B_{Z''}^{(i+1)}+\bM_{\pi'',Z''}\right)$, and
  \item $K_{Z''}+B_{Z''}^{(i+1)}+\bM_{\pi'',Z''}=(\tau'')^*\left(K_{Z'}+B_{Z'}^{(i+1)}+\bM_{\pi',Z'}\right)$.
\end{itemize}
Moreover, we infer that $p_{i}\bM_{\pi'}=p_{i}\bM_{\pi''}$ is b-Cartier and
$$p_{i}\left(K_{X'}+B_{\Phi_{i+1}}+\bM_{X'}\right)\sim p_{i}(\pi')^*\left(K_{Z'}+B_{Z'}^{(i+1)}+\bM_{\pi',Z'}\right).$$ 
We remark here that since $(X,B_{\Phi_{i+1}}+\bM)$ has an $\Rr$-complement which is klt over a neighborhood of $z$, by the canonical bundle formula, so are $(X',B'_{\Phi_{i+1}}+\bM)$ and $(Z'/Z\ni z,B_{Z'}^{(i+1)}+\bM_{\pi'})$.

Since $\psi:X\dashto X'$ is an MMP over $Z$ on $B-B_{\Phi_{i+1}}$ which is also an $-(K_X+B_{\Phi_{i+1}}+\bM_X)$-MMP over $Z$, we can see that
$$a(P,X',B_{\Phi_{i+1}}'+\bM)<a(P,X,B_{\Phi_{i+1}}+\bM)\le1$$
for any $\psi$-exceptional prime divisor $P$. Hence by \cite[Lemma 4.5]{BZ16} there exists a $\Qq$-factorial crepant model $(Y'/Z,B_{\Phi_{i+1},Y'}+\bM)$ of $(X'/Z,B_{\Phi_{i+1}}'+\bM)$ such that $Y'$ and $X$ are isomorphic in codimension one. It is clear that if $(Y'/Z\ni z,B_{\Phi_{i+1},Y'}+\bM)$ is $\cN_{i+1}$-complementary then so is $(X/Z\ni z,B_{\Phi_{i+1}}+\bM)$. Let $\tau''':(Z'''/Z,B^{(i+1)}_{Z'''}+\bM_{\pi'})\to (Z'/Z,B^{(i+1)}_{Z'}+\bM_{\pi'})$ be a crepant model of $(Z'/Z,B^{(i+1)}_{Z'}+\bM_{\pi'})$ such that any prime divisor $P_Y\subseteq\Supp B_{\Phi_{i+1},Y'}$ which is vertical over $Z'$, the image of $P_Y$ on $Z'''$ is a prime divisor. Since $(Z'''/Z\ni z,B^{(i+1)}_{Z'''}+\bM_{\pi'})$ is crepant to $(Z'/Z,B^{(i+1)}_{Z'}+\bM_{\pi'})$, one can see that $(Z'''/Z\ni z,B^{(i+1)}_{Z'''}+\bM_{\pi'})$ has an $\Rr$-complement which is klt over a neighborhood of $z$. By induction $(Z'''/Z\ni z,B^{(i+1)}_{Z'''}+\bM_{\pi'})$ is $n$-complementary for some $n\in\cN_{i+1}$. By Proposition \ref{prop:cbflift1}, $(Y'/Z\ni z,B_{\Phi_{i+1},Y'}+\bM)$ is also $n$-complementary. This completes the proof.
\end{proof}

\begin{proof}[Proof of Theorem \ref{thm:main}]
It follows from Theorem \ref{thm:klttype} by taking $I=1$ and $\Phi$ the standard set.
\end{proof}

\section{Boundedness of complements for $\Rr$-complementary generalized pairs}\label{sec9}
In this subsection, for the readers' convenience, we also give a short proof of \cite[Theorem 17]{Sho20}.

\begin{thm}[{\cite[Theorem 17]{Sho20}}]\label{thm:lctype}
Let $d, I$ be two positive integers, and $\I\subseteq[0,1]$ a subset such that $\I\cap\Qq$ is DCC. Then there exists a finite set $\cN$ of positive integers depending only on $d,I,$ and $\I$ satisfying the following.

Assume that $(X/Z\ni z,B)$ is a pair of dimension $d$ such that $B\in\I,$ $X$ is of Fano type over $Z$, and $(X/Z\ni z,B)$ is $\Rr$-complementary. Then $(X/Z\ni z,B)$ has an $n$-complement for some $n\in\cN$ such that $I\mid n$.
\end{thm}

\subsection{Strictly log canonical Calabi-Yau varieties}
\begin{defn}[cf.\ {\cite[\S11]{Sho20}}]
We say that a pair $(X/Z\ni z,B)$ is \emph{strictly lc Calabi-Yau} (\emph{strictly lc CY} for short) if
\begin{enumerate}
  \item $(X/Z\ni z,B)$ is an $\Rr$-complement of itself, and
  \item for any $\Rr$-complement $(X/Z\ni z,B^+)$ of $(X/Z\ni z,B)$, $B^+=B$ over a neighborhood of $z$.
\end{enumerate}
\end{defn}

\begin{rem}
When $\dim Z=\dim z=0$, $(X,B)$ is strictly lc CY if and only if $(1)$ holds.
\end{rem}

\begin{ex}
Let $\pi:X:=\Pp^1\times\Pp^1\to Z:=\Pp^1$, $z\in Z$ a closed point and $L_1,L_2$ are two sections. Then over a neighborhood of $z$, we have $(X,L_1+L_2)$ is lc and $K_X+L_1+L_2\sim_{\Rr,Z}0$. Since $K_X+L_1+L_2+\pi^*z\sim_{\Rr,Z}0$, $(X/Z\ni z,L_1+L_2)$ is not strictly lc CY.
\end{ex}

\begin{lem}\label{lem:crepantofmlccy}
Suppose that $(X/Z\ni z,B)$ is strictly lc CY and $\psi:X\dashto X'$ is a birational contraction over $Z$. Then $(X'/Z\ni z,B':=\psi_*B)$ is strictly lc CY.
\end{lem}
\begin{proof}
It follows from the definition and the fact that $(X/Z\ni z,B)$ and $(X'/Z\ni z,B')$ are crepant over a neighborhood of $z$.
\end{proof}

\begin{lem}\label{lem:mlcislccenter}
Suppose that $(X/Z\ni z,B)$ is a strictly lc CY pair with $\dim z<\dim Z$. Assume that $K_X+B\sim_{\Rr,Z}0$ and there is a g-pair $(Z,B_Z+\bM_{\pi})$ induced by a canonical bundle formula for $(X/Z,B)$ over $Z$. Then $\{\bar z\}$ is an lc center of $(Z,B_Z+\bM_{\pi})$.
\end{lem}
\begin{proof}
Suppose that the statement is not true. If $\dim z=\dim Z-1$, let $H:=\{\bar z\}$, otherwise let $H\ni z$ be any general ample divisor. Possibly shrinking $Z$ around $z$, we may assume that $H$ contains no lc center of $(Z,B_Z+\bM_\pi)$. Pick a small enough positive real number $\epsilon$, such that $(Z,B_Z+\epsilon H+\bM_\pi)$ is lc. Then $(X/Z\ni z,B+\epsilon\pi^*H)$ is an $\Rr$-complement of $(X/Z\ni z,B)$, where $\pi:X\to Z$. However, $B+\epsilon\pi^*H\neq B$ over any neighborhood of $z$. The lemma holds.
\end{proof}

\begin{prop}\label{prop:maxlcindex}
Let $d$ be a positive integer and $\I\subseteq[0,1]\cap\Qq$ a DCC set. Then there exists a positive integer $I$ depending only on $d$ and $\I$ satisfying the following.

Assume that $(X/Z\ni z,B)$ is a pair of dimension $d$ such that $B\in\I,$ $X$ is of Fano type over $Z$ , and $(X/Z\ni z,B)$ is strictly lc CY. Then $I(K_X+B)\sim0$ over some neighborhood of $z$.
\end{prop}
\begin{proof}
We first show that there exists a finite set $\I'\subseteq\I$ depending only on $d,\I$ such that $B\in\I'$ over some neighborhood of $z$. In fact, by Theorem \ref{thm:rctacc}, $\RCT(d,1,\I\cup\Zz_{\ge0})$ is an ACC set. It follows that 
$$\I':=\I\cap\RCT(d,1,\I\cup\Zz_{\ge0})$$ 
is a finite set as $\I$ is DCC. Since $(X/Z\ni z,B)$ is strictly lc CY, we know that $B\in\I'$ over some neighborhood of $z$.

According to \cite[Theorems 1.7 and 1.8]{Bir19} (see also \cite[Theorem 1.8]{HLS19}), we may find a positive integer $I$ which only depends on $d$ and $\I'$ such that $(X/Z\ni z,B)$ has a monotonic $I$-complement. We prove that $I$ has the required property. To see this, let $(X/Z\ni z,B+G)$ be an $I$-complement of $(X/Z\ni z,B)$ for some $\Rr$-divisor $G\ge0$. By assumption, $G=0$ over some neighborhood of $z$ which immediately implies that
$$I(K_X+B)=I(K_X+B+G)\sim0$$
over some neighborhood of $z$. The proof is finished.
\end{proof}

\subsection{Proof of Theorem \ref{thm:lctype}}
We first show Theorem \ref{thm:lctype} holds in the following situation.

\begin{prop}\label{prop:splctype}
Let $d$ and $I$ be two positive integers. Assume that $\cN$ is a finite set of positive integers given by Theorem \ref{thm:main} which only depends on $d$ and $I$. Then the following holds.

Assume that $(X/Z\ni z,B)$ is an $\Rr$-complementary pair of dimension $d$ and $X$ is of Fano type over $Z$. Assume that there is a contraction $\pi':X\to Z'$ over $Z$ and an open subset $U\subseteq Z'$, such that 
\begin{enumerate}
  \item $IB\in\Zz_{\ge0}$ over $U$, 
  \item $(X,B)$ is klt over $Z'\setminus U$, and
  \item there exists an $\Rr$-divisor $H'$ on $Z'$ which is ample over $Z$ such that
  $$-(K_X+B)\sim_{\Rr}(\pi')^*H'.$$ 
\end{enumerate} 
Then $(X/Z\ni z,B)$ is $\cN$-complementary.
\end{prop}
\begin{proof}
Possibly shrinking $Z$ near $z$, we may assume that $(X,B)$ is lc. Set $N:=\max_{n\in\cN}{n}$. We first claim that there is a boundary $B'$ on $X$ such that 
\begin{itemize}
 \item $(X,B')$ is klt, $K_X+B'\sim_{\Rr,Z}0$, and
 \item $B'\ge \frac{N}{N+1} B$ over $U$, and $B'\ge B$ over $Z'\setminus U$.
\end{itemize}
Suppose the claim is true now. Then
\begin{equation}\label{eqn: prop 7.9}
    \lf(n+1)B'\rf\ge n\lf B\rf+\lf(n+1)\{B\}\rf
\end{equation}
for any $n\in\cN$. By Theorem \ref{thm:main} and the construction of $\cN$, $(X/Z\ni z,B')$ is $n$-complementary for some $n\in\cN$. Thus $(X/Z\ni z,B)$ is $n$-complementary by \eqref{eqn: prop 7.9}.

Thus we only need to prove the claim. By assumption, we may find an effective $\Rr$-Cartier $\Rr$-divisor $H_1'\sim_{\Rr,Z}H'$ such that $Z'\setminus U\subseteq\Supp H_1'$ and $(X,B+H_1)$ is lc, where $H_1:=(\pi')^*H_1'$. In particular, we have 
$$K_X+B+H_1\sim_{\Rr,Z}0\text{ and }B+H_1>B\text{ over }Z'\setminus U.$$ 
Since $X$ is of Fano type over $Z$, one can find an $\Rr$-divisor $C$ such that $(X,C)$ is klt and $K_X+C\sim_{\Rr,Z}0$. We may pick a positive real number $\delta\in(0,1)$ satisfying that
\begin{itemize}
  \item $(1-\delta)(B+H_1)\ge B\text{ over }Z'\setminus U$, and
  \item $B':=(1-\delta)(B+H_1)+\delta C\ge \frac{N}{N+1}B\text{ over }U$.
\end{itemize}
It is obvious that $(X,B')$ is klt and $K_X+B'\sim_{\Rr,Z}0$. The proof is finished.
\end{proof}

\begin{proof}[Proof of Theorem \ref{thm:lctype}]
Let $I'$ be a positive integer given by Proposition \ref{prop:maxlcindex} which only depends on $d$ and $\Gamma\cap\Qq$. Possibly replacing $I$ with $II'$, we may assume that $I$ also satisfies the property of Proposition \ref{prop:maxlcindex}. Let $\Phi:=\Phi(\frac{1}{I}\Zz\cap[0,1])$ be the hyperstandard set associated to $\frac{1}{I}\Zz\cap[0,1]$. Let $\cN=\cN(d,I)$ be a finite set of positive integers given by Theorem \ref{thm:main}. We will show that $\cN$ has the required properties.

Possibly shrinking $Z$ near $z$ replacing $(X,B)$ by a dlt modification, we may assume that $(X,B)$ is $\Qq$-factorial dlt. Suppose that $(X/Z\ni z,B^+)$ is an $\Rr$-complement of $(X/Z\ni z,B)$. Possibly replacing $z$ by a closed point of $\bar{z}$ and shrinking $Z$ near $z$, we may assume that $z$ is a closed point, $(X,B^+)$ is lc, and $K_X+B^+\sim_{\Rr,Z}0$. Write
$$-(K_X+B_{\cN\_\Phi})\sim_{\Rr,Z} B^+-B_{\cN\_\Phi}=F+M,$$ 
where $F:=N_\sigma(B^+-B_{\cN\_\Phi}/Z)\ge0$ and $M:=B^+-B_{\cN\_\Phi}-F\ge0$
(cf. \cite[III, \S4]{Nak04}, \cite[\S3]{LX22}). Note that $F$ is well-defined as $B^+-B_{\cN\_\Phi}\ge0$. 

As $X$ is of Fano type over $Z$, we may run an MMP on $M$ over $Z$ which terminates with a good minimal model $X'$, such that $M'$ the strict transform of $M$ on $X'$ is semi-ample over $Z$. Since
$$M\sim_{\Rr,Z}-(K_X+B_{\cN\_\Phi}+F),$$
$X$ and $X'$ are isomorphic in codimension one by \cite[Lemma 2.4]{HX13}. We also see that $-(K_{X'}+B'_{\cN\_\Phi}+F')$ is semi-ample over $Z$ and thus induces a contraction $\pi':X'\to Z'$ over $Z$, where $B'_{\cN\_\Phi}$ and $F'$ are the strict transforms of $B_{\cN\_\Phi}$ and $F$ on $X'$ respectively. In particular, there is an effective $\Rr$-divisor $H'$ on $Z'$ which is ample over $Z$ such that 
$$-(K_{X'}+B'_{\cN\_\Phi}+F')\sim_{\Rr}(\pi')^*H'.$$ 
Note that $(X',B^{+\prime})$ is lc and thus $(X',B_{\cN\_\Phi}'+F')$ is also lc, where $B^{+\prime}$ is the strict transform of $B^+$ on $X'$.

We denote by $\eta'$ the generic point of $Z'$, and
$$\cZ_{scy}:=\{\eta'\} \cup \left\{z'\in Z'\mid \left(X'/Z'\ni z',B'_{\cN\_\Phi}+F'\right)\text{ is strictly lc Calabi-Yau}\right\}.$$
By Lemma \ref{lem:mlcislccenter} and \cite[Theorem 1.1]{AM06}, $\cZ_{scy}$ is a non-empty finite set.

\begin{claim}\label{claim:twofacts}
Over a neighborhood of $z'$ for any $z'\in\cZ_{scy}$, we have that
$$I\left(K_{X'}+B'_{\cN\_\Phi}+F'\right)\sim0.$$
\end{claim}
Suppose that Claim \ref{claim:twofacts} holds true now. Then by \cite[Lemma 2.5]{CHL23}, Lemma \ref{lem:mlcislccenter} and \cite[Theorem 1.1]{AM06}, we can find an open subset $U\subseteq Z'$ satisfies that 
\begin{itemize}
  \item $I\left(K_{X'}+B'_{\cN\_\Phi}+F'\right)\sim0\text{ over $U$, and }$
  \item $\left(X',B'_{\cN\_\Phi}+F'\right)\text{ is klt over }Z'\setminus U.$
\end{itemize}
In particular, $I(B'_{\cN\_\Phi}+F')\in\Zz_{\ge0}$ over $U$. Recall that 
$$-(K_{X'}+B'_{\cN\_\Phi}+F')\sim_{\Rr}(\pi')^*H'$$ 
where $H'$ is ample over $Z$. By Proposition \ref{prop:splctype}, $(X'/Z\ni z,B'_{\cN\_\Phi}+F')$ is $\cN$-complementary, and thus $(X'/Z\ni z,B')$ is also $\cN$-complementary by Lemma \ref{lem:N_Phicompl}. Furthermore, since $X$ and $X'$ are isomorphic in codimension one, $(X/Z\ni z,B)$ is $\cN$-complementary. This completes the proof.
\end{proof}
\begin{rem}
We note here that Shokurov wrote that Theorem \ref{thm:lctype} is also true in the generalized pairs setting \cite[Addendum 61]{Sho20}, however, it seems that the current proof is not applicable in this setting. The main reason is that in the proof of Theorem \ref{thm:lctype}, one needs to perturb the boundary of the pair which may change the nef part (in the generalized setting) and that may lead to failure. 
\end{rem}

\bibliographystyle{amsalpha}
\providecommand{\bysame}{\leavevmode\hbox to3em{\hrulefill}\thinspace}
\providecommand{\MR}{\relax\ifhmode\unskip\space\fi MR }
\providecommand{\MRhref}[2]{%
  \href{http://www.ams.org/mathscinet-getitem?mr=#1}{#2}
}

\end{document}